\documentclass[12pt, reqno, a4paper]{amsart}

\usepackage{amssymb}
\usepackage{color}
\definecolor{darkgreen}{rgb}{0,0.45,0} 
\usepackage[
colorlinks,citecolor=blue,linkcolor=blue, breaklinks]{hyperref}
\usepackage[matrix, arrow, curve]{xy}
\usepackage{bbm}
\usepackage{dsfont}
\usepackage{xcolor}

\usepackage[english]{babel}

\binoppenalty=\maxdimen
\relpenalty=\maxdimen

\theoremstyle{plain}
\newtheorem{theorem}{Theorem}[section]

\newtheorem{proposition}[theorem]{Proposition}

\theoremstyle{remark}
\newtheorem{remark}[theorem]{Remark}

\theoremstyle{definition}
\newtheorem{example}[theorem]{Example}

\newtheorem{definition}[theorem]{Definition}

\numberwithin{equation}{section}

\DeclareMathOperator{\Id}{Id}

\DeclareMathOperator{\id}{id}

\DeclareMathOperator{\Ker}{Ker}

\DeclareMathOperator{\End}{End}

\DeclareMathOperator{\sign}{sign}

\DeclareMathOperator{\Hom}{Hom}

\DeclareRobustCommand{\No}{\ifmmode{\nfss@text{\textnumero}}\else\textnumero\fi} 

\newbox\skewpullbackbox
\setbox\skewpullbackbox=\hbox{\xy 0;<1mm,0mm>: \POS(4,0)\ar@{-} (-4,0) \ar@{-} (8,4)
	\endxy}

\newbox\skwepullbackbox
\setbox\skwepullbackbox=\hbox{\xy 0;<1mm,0mm>: \POS(16,0)\ar@{-} (10,0) \ar@{-} (12,4)
	\endxy}

\newbox\ksewpullbackbox
\setbox\ksewpullbackbox=\hbox{\xy 0;<1mm,0mm>: \POS(0,-8)\ar@{-} (0,-4) \ar@{-} (4,-4)
	\endxy}

\newbox\pullbackbox
\setbox\pullbackbox=\hbox{\xy 0;<1mm,0mm>: \POS(4,0)\ar@{-} (0,0) \ar@{-} (4,4)
	\endxy}

\newbox\pullbackabox
\setbox\pullbackabox=\hbox{\xy 0;<1mm,0mm>: \POS(-4,-6)\ar@{-} (-8,-6) \ar@{-} (-4,-2)
	\endxy}
\newcommand{\pullbacka}{\copy\pullbackabox}

\newbox\pullbackbbox
\setbox\pullbackbbox=\hbox{\xy 0;<1mm,0mm>: \POS(-4,-4)\ar@{-} (-8,-4) \ar@{-} (-4,0)
	\endxy}

\newbox\pullbackcbox
\setbox\pullbackcbox=\hbox{\xy 0;<1mm,0mm>: \POS(4,-7)\ar@{-} (0,-7) \ar@{-} (4,-3)
	\endxy}

\newbox\pullbackdbox
\setbox\pullbackdbox=\hbox{\xy 0;<1mm,0mm>:\POS(-10,-6)\ar@{-} (-14,-6) \ar@{-} (-10,-2)
	\endxy}

\newbox\pushoutbox
\setbox\pushoutbox=\hbox{\xy 0;<1mm,0mm>: \POS(0,4)\ar@{-} (0,0) \ar@{-} (4,4)
	\endxy}

\newbox\pushoutabox
\setbox\pushoutabox=\hbox{\xy 0;<1mm,0mm>: \POS(2,8)\ar@{-} (2,4) \ar@{-} (8,8)
	\endxy}

\makeatletter
\@namedef{subjclassname@2020}{\textup{2020} Mathematics Subject Classification}
\makeatother

\begin{document}

\title[On a general notion of a polynomial identity and codimensions]{On a general notion of a polynomial identity and codimensions}

\author{A.\,S. Gordienko}

\email{alexey.gordienko@math.msu.ru}

\address{Department of Higher Algebra,
	Faculty of Mechanics and  Mathematics,
	M.\,V.~Lomonosov Moscow State University,
	Leninskiye Gory, d.1,  Main Building, GSP-1, 119991 Moskva, Russia }


\begin{abstract} Using the braided version of Lawvere's algebraic theories and Mac Lane's PROPs, we introduce polynomial identities
	for arbitrary algebraic structures in a braided monoidal category $\mathcal C$ as well as their codimensions
	in the case when $\mathcal C$ is linear over some field. The new cases include coalgebras, bialgebras, Hopf algebras, braided vector spaces, Yetter--Drinfel'd modules, etc.
		We find bases for polynomial identities and calculate codimensions in some important particular cases.

\end{abstract}


\thanks{The author is partially supported by the Russian Science Foundation grant \No\,22-11-00052}

\maketitle

\tableofcontents

\newpage

\section{Introduction}
Vector spaces $A$ with linear operations $A^{\otimes m} \to A^{\otimes n}$ where $m,n \in \mathbb Z_+$, e.g. coalgebras, bialgebras, Hopf algebras, braided vector spaces, Yetter--Drinfel'd modules, etc., find their applications in many areas of mathematics and physics. Such structures admit a universal approach via the notion of an \textit{$\Omega$-algebra} over a field
where $\Omega$ is an arbitrary set of symbols~$\omega$ of operations $\omega_A \colon A^{\otimes s(t)} \to A^{\otimes t(\omega)}$ of some arity $s(\omega)$ and coarity $t(\omega)$. (See an example of studying (co)actions on $\Omega$-algebras over fields e.g. in~\cite{AGV2}.) On the other hand, analogous objects, which will be called below \textit{$\Omega$-magmas}, can be introduced in any monoidal category.

Studying polynomial identities in an algebraic structure is an important aspect of studying the algebraic structure itself that leads to a discovery of new classes of such structures defined in terms of polynomial identities, which in turn may help in solving known problems on them.
Furthermore, when a basis for polynomial identities in a concrete algebra $A$ over a field is being calculated, certain numerical characteristics of polynomial identities, called codimensions, come into play naturally. The $n$th codimension $c_n(A)$ is just the dimension of the vector space of all linear maps $A^{\otimes n} \to A$ that can be realized by multilinear polynomials.
It is known~\cite[Section~6.2]{GiaZai} that there exists an tight relationship between the structure of~$A$ and the asymptotic behavior of  $c_n(A)$.

 A polynomial identity is traditionally introduced as a functional equality $f\equiv g$ of polynomials (or just monomials) $f,g$ in the corresponding signature where the polynomials $f,g$ are usually elements of the corresponding free algebraic structure. (We recall this in Section~\ref{SubsectionAlgebraicStructures} below.) Unfortunately, this approach may not work for $\Omega$-algebras with $t(\omega)\geqslant 2$ for some $\omega \in \Omega$, since, say, free coalgebras do not exist. (There exist only cofree coalgebras, see~\cite{SweedlerBook}.) In order to overcome this difficulty,  Mikhail Kochetov introduced~\cite{KochetovCoalgId} in 2000 the notion of a polynomial identity for a coalgebra $C$ as a polynomial identity for the algebra $C^*$ dual to~$C$. The same idea was used even earlier, in 1994,  by J.~Anquella,  T.~Cort\'es and F.~Montaner
 to introduce varieties of coalgebras~\cite{AnqCortMont}. However this did not answer the question what would be a polynomial identity simultaneously involving all possible operations in, say, a Hopf algebra, which is both an algebra and a coalgebra.

Analyzing the definition of a polynomial identity in a traditional algebraic structure $A$ (in the category $\mathbf{Sets}$ of small sets) of an arbitrary signature $\Omega$, we see in Section~\ref{SubsectionAlgebraicStructures} below that every $\Omega$-monomial can be understood as a composition of arrows corresponding to symbols of $\Omega$, swaps and operations $\Delta$ and $\varepsilon$ defining the unique comonoid structure on the set $A$. It turns out that $\Delta$ and $\varepsilon$ are needed to consider ``non-multilinear'' polynomial identities, i.e. where a letter occurs in one of the sides more than once or in only one side of the identity. All of this inspires to carry out Lawvere's approach to algebraic structures~\cite{Lawvere}, which establishes a 1-1 correspondence between algebraic structures
of a fixed signature $\Omega$ and product preserving functors from a certain category $\mathbb{A}$, called an \textit{algebraic theory}, to $\mathbf{Sets}$. Here $\Omega$-monomials can be viewed just as morphisms in~$\mathbb{A}$. Moreover, identifying certain morphisms in $\mathbb{A}$, one can describe in this way varieties of algebraic structures, i.e. classes of algebraic structures satisfying a fixed set of polynomial identities. In 1965, Mac Lane considered~\cite{MacLane65} symmetric monoidal analogs of algebraic theories, which he called PROPs (= ``product and permutation categories'').
In 1996 Martin Markl used linear PROPs (which he briefly referred to as just ``theories'') to study deformations of arbitrary $\Omega$-magmas~\cite{MarklPROPDef}. In 2002 Teimuraz Pirashvili described explicitly the PROP corresponding to the notion of a bialgebra over a field~\cite{PirashviliPROPBialgebra}. Using the terminology of the ``traditional'' PI-theory, Pirashvili's result can be compared with a description of a relatively free algebra for some variety of algebras.
 
 In the current article we use braided strict monoidal categories that we call \textit{braided monoidal algebraic theories} (BMATs for short). BMATs are just braided versions of PROPs.
Given a set $\Omega$ together with maps $s,t \colon \Omega \to \mathbb Z_+$, we describe explicitly a BMAT $\mathcal M(\Omega)$ such that for every braided monoidal category $\mathcal C$ there exists a 1-1 correspondence between $\Omega$-magmas in $\mathcal C$ and braided strong monoidal functors $\mathcal M(\Omega) \to \mathcal C$. The morphisms in $\mathcal M(\Omega)$
 are called \textit{$\Omega$-monomials}. When hom-sets in $\mathcal C$ are vector spaces over a field $\mathbbm{k}$, we define the category
$\mathcal P(\mathbbm{k}, \Omega)$ of $\Omega$-polynomials with coefficients in $\mathbbm{k}$ and codimensions of polynomial identities. Varieties of $\Omega$-magmas are introduced in a natural way.

It turns out that if one considers $\Omega$ corresponding to a coalgebra, then $\Omega$-polynomial identities with coefficients in $\mathbbm{k}$ in a coalgebra $C$ over $\mathbbm{k}$, having the domain $1$, can be identified with multilinear polynomial identities in $C$ introduced by Mikhail Kochetov in~\cite{KochetovCoalgId}.  There are two reasons why non-multilinear polynomial identities in coalgebras cannot fit into the framework of BMATs at all. First, as we have already mentioned, when we consider non-multilinear polynomial identities in algebras over $\mathbbm{k}$, we must actually use maps $\Delta$ and $\varepsilon$ from the comonoid structure in~$\mathbf{Sets}$, which are non-linear, so we drop out of the category $\mathbf{Vect}_\mathbbm{k}$ of $\mathbbm{k}$-vector spaces. 
Second, the procedure of making the algebra $C^*$ out of a coalgebra $C$ is essentially linear, so it is hard to expect that it would establish a correspondence between anything non-linear. On the other hand, over a field of characteristic $0$ all polynomial identities in an algebra are consequences of its multilinear polynomial identities~\cite[Theorem 1.3.8]{GiaZai}.

If an $\Omega$-algebra $A$ over a field $\mathbbm k$ is endowed with an additional structure, e.g. a group grading or a structure of a (co)module algebra over a Hopf algebra, it is natural to include this additional structure in the signature of polynomial identities~\cite{AljaKassel, BahturinLinchenko, BereleHopf, GiaZai, ASGordienko15}.

On the one hand, if $A$ is an $H$-module for a unital associative algebra $H$, e.g. $A$ is endowed with a generalized $H$-action
(we recall the definition at the end of Section~\ref{SubsectionPIHmod}), one can just add the symbols of operators from $H$ to the signature $\Omega$ and consider polynomial $\Omega \sqcup H$-identities. On the other hand, if $H$ is a Hopf algebra,
the categories ${}_H\mathsf{Mod}$ and $\mathsf{Comod}^H$ of left $H$-modules and right $H$-comodules, respectively,
are monoidal. If $G$ is a group, then the category of $G$-graded vector spaces is monoidal too.
Finally, a graded or a (co)module $\Omega$-algebra $A$ is just an $\Omega$-magma in the category $\mathcal C$ of graded vector spaces or (co)modules, respectively, which prompts to take $\mathcal C$ as the base category and try to define polynomial identities
in $A$ in some intrinsic way.  However if we consider the category $\mathcal P(\mathbbm k, \Omega)$ of $\Omega$-polynomials with coefficients in $\mathbbm k$, the images of morphisms in $\mathcal P(\mathbbm k, \Omega)$ under $\mathcal E_A$ will be the linear maps $A^{{}\otimes m} \to A^{{}\otimes n}$ resulting only from the operations $\omega_A$, $\omega \in \Omega$, but not from the additional structure. In order to indeed include the additional structure in the signature of polynomial identities,
in Section~\ref{SectionPIGr(Co)mod} we use the reconstruction technique (see e.g. \cite[Chapter 5]{EGNObook}).
In this way for $\mathcal C = {}_H\mathsf{Mod}$ we get all multilinear polynomial $H$-identities.
For $\mathcal C = \mathsf{Comod}^H$ we in fact recover multilinear polynomial $H^*$-identities where the $H^*$-action is induced by the $H$-comodule structure. However, if the Hopf algebra $H$ is finite dimensional, the polynomial $\Omega \sqcup H^*$-identities obtained can be identified with multilinear polynomial $H$-identities in $H$-comodule algebras. 

In Section~\ref{SectionExamplesApplications} we study polynomial identities and their codimensions in vector spaces, (co)commutative Hopf algebras and Yetter~--- Drinfeld modules.


\section{$\Omega$-monomials and polynomial identities in $\Omega$-magmas}

\subsection{Monoids, comonoids and Hopf monoids}\label{SubsectionMonComonHopfMon}

We refer the reader to e.g.~\cite{EGNObook} for an account of monoidal categories.

Let $\mathcal C$ be a monoidal category with a monoidal product $\otimes$ and a monoidal unit $\mathbbm{1}$.
Recall that a \textit{monoid}  $(M,\mu, u)$ in $\mathcal C$ is an object $M$ together with morphisms
$\mu \colon M \otimes M \to M$ and $u \colon \mathbbm{1} \to M$ making the diagrams below commutative:
$$ \xymatrix{(M \otimes M) \otimes M \ar[rr]^{\sim} \ar[d]^{\mu \otimes \id_M} & &  M \otimes (M \otimes M) \ar[d]^{\id_M \otimes \mu} \\
	M \otimes M \ar[rd]^\mu & & M \otimes M \ar[ld]_\mu \\
	& M & 
} \qquad \xymatrix{ M  \ar[r]^(0.4)\sim  \ar[d]_\sim \ar@{=}[rdd]      &     M \otimes \mathbbm{1} \ar[d]^{\id_M\otimes u} \\
\mathbbm{1} \otimes M  \ar[d]_{u\otimes \id_M}  &     M \otimes M \ar[d]^\mu \\
 M \otimes M   \ar[r]^(0.55)\mu               &     M
} 
$$

\begin{example}
	Recall that the category $\mathbf{Sets}$ of small sets is a monoidal category with the Cartesian monoidal product $\times$
	and the monoidal unit $\lbrace * \rbrace$, a single element set. Monoids  in $\mathbf{Sets}$
	are just ordinary set-theoretical monoids $M$ where $\mu(m,n) = mn$ for every $m,n\in M$ and $u(*)= 1_M$.
\end{example}	

\begin{example} The category $\mathbf{Vect}_\mathbbm{k}$
	of vector spaces over a field $\mathbbm{k}$ is monoidal too where the monoidal product is the tensor product $\otimes$
	and the monoidal unit is the base field $\mathbbm{k}$. Monoids in $\mathbf{Vect}_\mathbbm{k}$ are just unital associative algebras over the field $\mathbbm{k}$.
\end{example}	

A comonoid in a monoidal category $\mathcal C$ is a monoid  $(C,\Delta, \varepsilon)$ in the dual category $\mathcal C^\mathrm{op}$. The corresponding morphisms $\Delta \colon C \to C \otimes C$
and $\varepsilon \colon C \to \mathbbm{1}$ are called the \textit{comultiplication}
and the \textit{counit}, respectively.

Comonoids in $\mathbf{Vect}_\mathbbm{k}$ are called \textit{coalgebras} over the field $\mathbbm{k}$.

All comonoids in $\mathbf{Sets}$ are trivial. Namely, every set $X$ admits exactly one structure of a comonoid
where the comultiplication $\Delta \colon X \to X\times X$ is the diagonal map, i.e. $\Delta x = (x,x)$ for all $x\in X$, and the counit $\varepsilon \colon X \to \lbrace * \rbrace$ is trivial, i.e. $\varepsilon(x) = *$ for all $x\in X$.

Recall that if the category $\mathcal C$ is braided, then $\mathsf{Mon}(\mathcal C)$ is a monoidal category too.
Objects of the category $\mathsf{Comon}(\mathsf{Mon}(\mathcal C))$ (which is isomorphic to $\mathsf{Mon}(\mathsf{Comon}(\mathcal C))$) are called \textit{bimonoids} in $\mathcal C$.

If $(A,\mu_A, u_A)$ is a monoid and $(C,\Delta_C,\varepsilon_C)$ is a comonoid in a monoidal category $\mathcal C$,
then the set $\mathcal C (C,A)$ of all morphisms $C\to A$ in $\mathcal C$ admits a structure of an ordinary monoid:
the multiplication is defined by
$$\varphi * \psi := \mu_A (\varphi \otimes \psi)\Delta_C \text{ for all } \varphi,\psi \in \mathcal C (C,A)$$
and $u_A\varepsilon_C$ is the identity element.
The monoid $\mathcal C (C,A)$ is called the \textit{convolution monoid}.

A bimonoid $H$ in a braided monoidal category $\mathcal C$ is called a \textit{Hopf monoid}
if $\id_H \in \mathcal C (H,H)$ admits an inverse $S \colon H \to H$, which is called the \textit{antipode}. We denote the category of Hopf monoids in $\mathcal C$ by $\mathsf{Hopf}(\mathcal C)$.

\begin{example}
	Hopf monoids in $\mathbf{Vect}_\mathbbm{k}$, where $\mathbbm{k}$ is a field, are exactly Hopf algebras over~$\mathbbm{k}$.
\end{example}

\begin{example}	If $G$ is a group, then the group algebra $\mathbbm{k} G$ is a Hopf algebra where $\Delta g := g\otimes g$, $\varepsilon(g)=1$, $Sg= g^{-1}$ for all $g\in G$.
\end{example}

\begin{example}
	Hopf monoids in $\mathbf{Sets}$ are exactly groups.
\end{example}

Denote by $\tau_{X,Y} \colon X \otimes Y \mathrel{\widetilde\to} Y \otimes X$ the braiding in $\mathcal C$.

For a monoid $(A,\mu_A, u_A)$ denote by $A^{\mathrm{op}}$
the monoid  $(A,\mu_A \tau_{A,A}, u_A)$. Analogously, for a comonoid $(C,\Delta_C,\varepsilon_C)$ denote by $C^{\mathrm{cop}}$ the comonoid  $(C,\tau_{C,C}\Delta_C,\varepsilon_C)$.

Standard convolution techniques (see e.g.~\cite[Section 4.2]{DNR} or~\cite[Lemma 35, Proposition 36]{Porst1}) combined with a diagram chasing show that $S \colon H \to H^\mathrm{op}$ is a monoid homomorphism and $S \colon H^\mathrm{cop} \to H$ is a comonoid homomorphism.
Moreover, every bimonoid homomorphism of Hopf monoids commutes with (or preserves) the antipode.

\subsection{Algebraic structures of a given signature}\label{SubsectionAlgebraicStructures} Before considering different generalizations,
below we recall a classical definition of an algebraic structure $A$ of a given signature $\Omega$ and polynomial identities in $A$ and notice that the trivial comonoid structure on $A$ is present in the constructions implicitly.

Let $\Omega$ be a set with a map $s \colon \Omega\to \mathbb Z_+$. Recall that an \textit{algebraic structure of signature $\Omega$}
is a set $A$ endowed with maps $\omega_A \colon A^{s(\omega)} \to A$ for every $\omega \in \Omega$, which are called \textit{$s(\omega)$-ary operations}. If $A$ and $B$ are algebraic structures of the same signature $\Omega$, then
a map $f\colon A \to B$ is called a \textit{homomorphism} if $$f\bigl(\omega_A(a_1, \ldots, a_{s(\omega)})\bigr) = 
\omega_B\bigl(f(a_1), \ldots, f(a_{s(\omega)})\bigr)\text{ for all }\omega \in \Omega\text{ and  }a_1, \ldots, a_{s(\omega)} \in A.$$

\begin{example}
	A group can be treated as an algebraic structure of signature $\Omega = \left\lbrace \cdot, 1, (-)^{-1} \right\rbrace$
	where $s(\cdot)=2$, $s(1)=0$ and $s\left((-)^{-1}\right)=1$.
\end{example}
\begin{example}
	A vector space over a field $\mathbbm k$ can be treated as an algebraic structure of signature $\Omega = \mathbbm k \sqcup \left\lbrace  +, 0, - \right\rbrace$
	where $s(\lambda)=1$ for all $\lambda \in \mathbbm k $, $s(+)=2$, $s(0)=0$ and $s\left(-\right)=1$.
\end{example}

\textit{(Absolutely) free} algebraic structures $\mathcal F_\Omega (X)$ are constructed inductively: given a set $X$, let
\begin{enumerate}
	\item  $X \subseteq \mathcal F_\Omega (X)$;
	\item  $\omega(w_1, \ldots, w_{s(\omega)}) \in \mathcal F_\Omega (X)$ for all $\omega \in  \Omega$ and all 
	$w_1, \ldots, w_{s(\omega)}\in \mathcal F_\Omega (X)$.
\end{enumerate}
The operations $\omega_{\mathcal F_\Omega (X)} \colon \mathcal F_\Omega (X)^{s(\omega)} \to  \mathcal F_\Omega (X)$ are defined as
follows: $$\omega_{\mathcal F_\Omega (X)}(w_1, \ldots, w_{s(\omega)}):= \omega(w_1, \ldots, w_{s(\omega)})
\text{ for all }w_1, \ldots, w_{s(\omega)}\in \mathcal F_\Omega (X).$$

Then for every algebraic structure $A$ of signature $\Omega$ and every map $f_0 \colon X \to A$ there exists a unique homomorphism $f \colon \mathcal F_\Omega (X) \to A$ such that $f\iota = f_0$ where $\iota$ is the embedding $ X \hookrightarrow \mathcal F_\Omega (X)$. By this reason the elements of $\mathcal F_\Omega (X)$ are called \textit{derived operations} and $f(w)$
is the \textit{value} of $w\in \mathcal F_\Omega (X)$ under the substitution $f_0$. By $w_A$ 
we denote the map $A^{m} \to A$ induced by each $w\in \mathcal F_\Omega (X)$ where $m$ is the number of letters $x_1, \ldots, x_m \in X$ actually appearing in $w$ and $w_A(a_1, \ldots, a_m):=f(w)$ for every $a_1, \ldots, a_m \in A$ where $f_0 \colon X \to A$
is some map such that $f_0(x_i) = a_i$ for every $1\leqslant i \leqslant m$.

An expression $u \equiv v$ where $u,v \in \mathcal F_\Omega (X)$ is called a \textit{polynomial identity} in  an algebraic structure $A$ of signature $\Omega$ if $f(u)=f(v)$ for all maps $f_0 \colon X \to A$. In other words, $u \equiv v$ is a polynomial identity in $A$
if $u_A=v_A$ as maps. (If some letter from $X$ appears only in one of the expressions $u$ and $v$, then we add a fictitious variable to the map corresponding to the other expression.)

Note that maps $w_A$, where $w\in \mathcal F_\Omega (X)$, are compositions of maps

\begin{enumerate}
	\item $(\id_A)^k \times \omega_A \times (\id_A)^\ell$ where $\omega \in \Omega$;
	\item $(\id_A)^k \times \tau \times (\id_A)^\ell$ where $\tau \colon A^2 \to A^2$ is the swap: $\tau(a,b):=(b,a)$ for all $a,b \in A$;
	\item $(\id_A)^k \times \Delta \times (\id_A)^\ell$ where $\Delta \colon A \to A^2$ is the diagonal map: $\Delta(a):=(a,a)$ for all $a \in A$;
	\item $(\id_A)^k \times \varepsilon \times (\id_A)^\ell$ where $\varepsilon$ is the unique map from $A$ to the one element set $\lbrace * \rbrace$;
	\item $(\id_A)^k \times \alpha \times (\id_A)^\ell$  where $\alpha$ is one of the identifications $\lbrace * \rbrace \times A \mathrel{\widetilde\to} A$ and	$ A \times \lbrace * \rbrace  \mathrel{\widetilde\to} A$.
\end{enumerate}
(The numbers $k,\ell \in \mathbb Z_+$ are arbitrary.)

	The last two types of maps are needed to be able to include fictitious variables.
	
	Recall that the category $\mathbf{Sets}$ is symmetric with the swap $\tau$ and the maps $\Delta$ and $\varepsilon$
	defined above coincide with those from the unique structure of a comonoid in $\mathbf{Sets}$ on the set $A$. All of this makes it natural to transfer the definition of a polynomial identity to similar structures in an arbitrary braided monoidal category.
	
	Before making the formal definition, we notice that in this sense it is reasonable to treat the comonoid maps $\Delta$ and $\varepsilon$ as the part of an algebraic structure $A$ that allows ``non-multilinear'' polynomial identities.

\subsection{$\Omega$-magmas}\label{SubsectionOmegaMagmas}

Let $\Omega$ be a set together with maps $s,t \colon \Omega \to \mathbb Z_+$. We will refer to $\Omega$ as the \textit{signature} too. 

\begin{definition} An \textit{$\Omega$-magma} in a monoidal category $\mathcal C$ is 
	an object $A$ endowed with morphisms $\omega_A \colon A^{\otimes s(\omega)} \to
	A^{\otimes t(\omega)}$ for every $\omega \in \Omega$. Here we use the convention that $A^{\otimes 0} := \mathbbm{1}$, the monoidal unit in $\mathcal C$.
\end{definition}

\begin{example}\label{ExampleMagma}
	Every magma (i.e., a set with a binary operation) is the same as an $\Omega$-magma in the category $\mathbf{Sets}$ of sets where $\Omega=\lbrace\mu \rbrace$ with $s(\mu)=2$ and $t(\mu)=1$.
\end{example}

\begin{example}\label{ExampleAlgebraOmega}
	Every algebra (not necessarily associative or unital)  over a field $\mathbbm{k}$
	is the same as an $\Omega$-magma in the category $\mathbf{Vect}_\mathbbm{k}$ of $\mathbbm k$-vector spaces where $\Omega=\lbrace\mu \rbrace$ with $s(\mu)=2$ and $t(\mu)=1$.
\end{example}

\begin{example}\label{ExampleUnitalAlgebraOmega}
	Every unital algebra $A$ over a field $\mathbbm{k}$ is an example of an $\Omega$-magma in $\mathbf{Vect}_\mathbbm{k}$	for $\Omega=\lbrace\mu, u \rbrace$, $s(\mu)=2$, $t(\mu)=1$, $s(u)=0$, $t(u)=1$, where $u_A \colon \mathbbm{k} \to A$ is defined by $u_A(\alpha)=\alpha 1_A$
	for $\alpha \in \mathbbm{k}$. An ordinary monoid is an example of an $\Omega$-magma in $\mathbf{Sets}$ for the same $\Omega$.
\end{example}		

\begin{example}\label{ExampleCoalgebraOmega}
	Every coalgebra $C$ over a field $\mathbbm{k}$ is an example of an $\Omega$-magma in $\mathbf{Vect}_\mathbbm{k}$	for $\Omega=\lbrace\Delta, \varepsilon \rbrace$, $s(\Delta)=1$, $t(\Delta)=2$, $s(\varepsilon)=1$, $t(\varepsilon)=0$.
\end{example}	

In general, $\Omega$-magmas in $\mathbf{Vect}_\mathbbm{k}$ are called \textit{$\Omega$-algebras} over $\mathbbm{k}$~\cite{AGV2}.

\begin{example}\label{ExampleBraidingOmega}
	An object $A$ endowed with a braiding $\sigma_A \colon A \otimes A \to A \otimes A$
	is an example of an $\Omega$-magma for $\Omega=\lbrace \sigma \rbrace$, $s(\sigma)=2$, $t(\sigma)=2$.
\end{example}		

Let $A$ and $B$ be $\Omega$-magmas in a monoidal category $\mathcal C$. A morphism $f\colon A \to B$ is called an \textit{$\Omega$-magma homomorphism}
if for every $\omega \in \Omega$ the diagram below is commutative:
$$
\xymatrix{
	A^{\otimes s(\omega)} \ar[rr]^-{f^{\otimes s(\omega)}} \ar[d]_{\omega_A} && B^{\otimes s(\omega)} \ar[d]^{\omega_B}\\
	A^{\otimes t(\omega)} \ar[rr]^-{f^{\otimes t(\omega)}}  && B^{\otimes t(\omega)} 
}
$$

Denote by $\Omega\text{-}\mathsf{Magma}(\mathcal C)$ the category of $\Omega$-magmas in $\mathcal C$.

\subsection{Braided monoidal algebraic theories}\label{SubsectionBMAT}

A \textit{braided monoidal algebraic theory} (BMAT for short) is a braided strict monoidal category $\mathbb{A}$
where the objects are non-negative integers $n\in \mathbb Z_+$ and $m\otimes n = m+n$ for all $m, n\in \mathbb Z_+$.
(Hence $0$ is automatically the monoidal unit.)
When $\mathbb{A}$ is symmetric, $\mathbb{A}$ is just a PROP in the sense of Mac Lane~\cite{MacLane65}. When the monoidal product in $\mathbb{A}$ is the categorical product, i.e. $\mathbb{A}$ is Cartesian, then $\mathbb{A}$ is just an algebraic theory in the sense of Lawvere~\cite{Lawvere}.

Fix a set $\Omega$ and maps $s,t \colon \Omega \to \mathbb Z_+$. We need to introduce expressions in the signature $\Omega$. In fact, they will be just morphisms in the BMAT~$\mathcal M(\Omega)$ defined below.

Consider the free monoid $(M, \bullet)$ with the set of free generators $$\Omega \sqcup \lbrace \id_m, \tau_{m,n}, \tau^{-1}_{m,n} \mid m,n \in \mathbb Z_+ \rbrace$$ and
define monoid homomorphisms $s \colon M \to\mathbb  Z_+$
and $t \colon M \to \mathbb Z_+$
extending maps $s, t \colon \Omega \to \mathbb Z_+$, such that
\begin{enumerate}
\item $s(\id_m)=t(\id_m)=m$;
\item $s(\tau_{m,n})=t(\tau_{m,n})=s(\tau^{-1}_{m,n})=t(\tau^{-1}_{m,n})=m+n$.
\end{enumerate}

Now consider the directed graph $\Gamma$ with the set of vertices $\mathbb Z_+$ and the set of edges $M$
where an edge $w \in M$ goes from the vertex $s(w)$ to the vertex $t(w)$.

Recall that an equivalence relation $\sim$ on hom-sets $X (a,b)$ of a category $X$ is a \textit{congruence} if
$g_1 f_1 \sim g_2 f_2$ for all $f_1, f_2 \in X (a,b)$, $g_1, g_2 \in X (b,c)$
and objects $a,b,c$ such that $f_1 \sim f_2$ and $g_1 \sim g_2$.
The \textit{factor category} $X/{\sim}$ is the category with the same objects as in $X$ and the hom-sets $X (a,b)/{\sim}$.

Let $X$ be the category where the objects are non-negative integers and the morphisms are finite paths
in $\Gamma$.
If $m,n \in \mathbb Z_+$, define $m \otimes n :=  m+n$.
 If $w_1 \ldots w_s \in X(m,n)$ and  $w_1'\ldots w_t' \in X(m',n')$
  for some $m,n,m',n' \in \mathbb Z_+$ and $w_1, \ldots, w_s, w_1', \ldots, w_t' \in M$, then define 
$w_1 \ldots w_s \otimes w_1' \ldots w_t' \in X(m+m',n+n')$ 
as $$(w_1 \bullet  \id_{n'}) \dots (w_s \bullet  \id_{n'})
 (\id_m \bullet  w_1') \ldots (\id_m \bullet  w_t').$$

Denote by $\mathcal M(\Omega)$ the factor category $X/{\sim}$ where  $\sim$ is the minimal congruence making 
\begin{enumerate}
\item $\id_m$ the identity morphisms of objects $m$;
\item $\tau^{-1}_{m,n}$ the inverses of $\tau_{m,n}$;
\item $X/{\sim}$ a braided strict monoidal category with the braidings $\tau_{m,n}$,
the monoidal unit~$0$ and the monoidal product $\otimes$.
\end{enumerate}

Then $\mathcal M(\Omega)$ is a BMAT, which we call the \textit{category of $\Omega$-monomials}.
The morphisms in $\mathcal M(\Omega)$ are called \textit{$\Omega$-monomials}.
 Define $s(f):= m$ and $t(f):= n$ for $f \in \mathcal M(\Omega)(m,n)$. Note that $1$ is an $\Omega$-magma in $\mathcal M(\Omega)$.
 
 \begin{remark}
 	Denote by $\mathbf{BMAT}$ the category of BMATs where the morphisms are braided strict monoidal functors and by $U \colon \mathbf{BMAT} \to ((\mathbb Z_+ \times \mathbb Z_+) \downarrow \mathbf{Sets})$ the functor
 	that assigns every BMAT $\mathbb A$ the function $(m,n) \mapsto \mathbb A(m,n)$, $m,n\in \mathbb Z_+$. We may regard $U$ as a forgetful functor, since $U\mathbb A $ is just the collection of hom-sets of $\mathbb A$ without any structure on them.
 	Then we get a bijection
 	$$\mathbf{BMAT}(\mathcal M(\Omega), \mathbb A) \mathbin{\widetilde\to} ((\mathbb Z_+ \times \mathbb Z_+) \downarrow \mathbf{Sets})(\Omega, U \mathbb A)$$
 	natural in $\Omega$ and $\mathbb A$. Here we treat the signature $\Omega$ as a function $\mathbb Z_+ \times \mathbb Z_+ \to \mathbf{Sets}$
 	where $\Omega(m,n) := \lbrace \omega \in \Omega \mid s(\omega)=m,\ t(\omega)=n \rbrace$. In this sense the functor $\mathcal M(-)$ is indeed the free functor left adjoint to the forgetful functor $U$.
 \end{remark}

Given braided monoidal categories $\mathcal C$ and $\mathcal D$, denote by $\mathbf{BSMF}(\mathcal C, \mathcal D)$ the category of braided strong monoidal functors $\mathcal C \to \mathcal D$ where morphisms between $F,G \colon \mathcal C \to \mathcal D$
are all monoidal transformations $\alpha \colon F \Rightarrow G$.

\begin{proposition}\label{PropositionOmegaMagmaBSMFunctorsEquivalence} Let $\Omega$ be a set together with maps $s,t \colon \Omega \to \mathbb Z_+$
	and let $\mathcal C$ be a braided monoidal category. Then the categories $\Omega\text{-}\mathsf{Magma}(\mathcal C)$ and $\mathbf{BSMF}(\mathcal M(\Omega), \mathcal C)$ are equivalent.
\end{proposition}
\begin{proof}
Let $A$ be an arbitrary $\Omega$-magma in $\mathcal C$. Denote by $\mathcal E_A$ a braided strong monoidal functor $\mathcal M(\Omega) \to \mathcal C$ mapping
\begin{enumerate}
	\item  $1$ to $A$;
	\item $\omega$ to $\omega_A$ for every $\omega \in \Omega$.
\end{enumerate}

Note that the functor $\mathcal E_A$ is completely determined by the choice of the natural isomorphism $\mathcal E_A(m)\otimes \mathcal E_A(n) \mathrel{\widetilde\to} \mathcal E_A(m+n)$ and an isomorphism $\mathbbm{1} \mathrel{\widetilde\to} \mathcal E_A(0)$ where $\mathbbm{1}$ is the monoidal unit in $\mathcal C$. Therefore, $\mathcal E_A$ is unique up to a monoidal isomorphism
and we get a functor $\Omega\text{-}\mathsf{Magma}(\mathcal C) \to \mathbf{BSMF}(\mathcal M(\Omega), \mathcal C)$, $A \mapsto \mathcal E_A$.

Conversely, every braided strong monoidal functor $F \colon\mathcal M(\Omega) \to \mathcal C$ defines the $\Omega$-magma $F(1)$
with the operations $\omega_{F(1)}$ that are equal to $F(\omega)$ composed with the corresponding isomorphisms $F(m)\mathrel{\widetilde\to} F(1)^{\otimes m}$, $\omega \in \Omega$. Given braided strong monoidal functors $F,G \colon\mathcal M(\Omega) \to \mathcal C$, the map $\alpha \mapsto \alpha_1$ defines  a one-to-one correspondence between monoidal natural transformations $\alpha \colon F \Rightarrow G$
and $\Omega$-magma homomorphisms $f \colon F(1) \to G(1)$.
\end{proof}

\subsection{Polynomial identities in $\Omega$-magmas}\label{SubsectionOmegaIndentities}
Again fix a set $\Omega$ and maps $s,t \colon \Omega \to \mathbb Z_+$.
Let $f$ and $g$ be $\Omega$-monomials such that $s(f)=s(g)$
and $t(f)=t(g)$. Then the expression $f\equiv g$ is called a \textit{$\Omega$-polynomial identity}.
(Here we use the word ``polynomial'' since $f\equiv g$ contains two monomials $f$ and $g$.)

Let $A$ be an $\Omega$-magma in a braided monoidal category~$\mathcal C$.
We say that  $f\equiv g$ is a \textit{polynomial identity in $A$} or that $f\equiv g$ \textit{holds in $A$} if  $\mathcal E_A(f)=\mathcal E_A(g)$. Polynomial identities $f\equiv g$ in $A$ can be identified with morphisms $(f,g)$ in the category $\Id(A)$ where
the set of objects is $\mathbb Z_+$ and $\Id(A)(m,n)$  for $m,n \in \mathbb Z_+$
is defined by the pullback
$$ \xymatrix{
\Id(A)(m,n) \ar[d] \ar[r] \ar@{}[rd]|<{\pullbacka} & \mathcal M (\Omega)(m,n) \ar[d]^{\mathcal E_A} \\
\mathcal M (\Omega)(m,n) \ar[r]^{\mathcal E_A} & \mathcal C (m,n)\\
}
$$

Let $V$ be some set of $\Omega$-polynomial identities. The full subcategory $\mathop\mathrm{Var}(V)$ of $\Omega\text{-}\mathsf{Magma}(\mathcal C)$, consisting of all $\Omega$-magmas satisfying $V$, is called a \textit{variety} of $\Omega$-magmas.  

\begin{example}
Coalgebras and Hopf algebras form varieties for the corresponding sets of polynomial identities.
\end{example}

  When studying polynomial identities in an $\Omega$-magma that already lies in some variety $\mathop\mathrm{Var}(V)$, it is natural to identify with each other those polynomial identities that follow from each other by virtue of $V$.
For example, in the variety of associative algebras it is natural to identify $x(yz)\equiv x(zy)$ and $(xy)z\equiv (xz)y$.
Below we describe the corresponding construction for $\Omega$-magmas in an arbitrary braided monoidal category~$\mathcal C$.

If $\sim$ is a congruence on a monoidal category $X$ and 
$f_1 \otimes g_1 \sim f_2 \otimes g_2$ for all $f_1, f_2 \in X (a,b)$, $g_1, g_2 \in X (c,d)$
and objects $a,b,c,d$ such that $f_1 \sim f_2$ and $g_1 \sim g_2$,
then we say that the congruence $\sim$ is \textit{monoidal}. It is easy to see that in this case $X/{\sim}$ induces from $X$ the structure of a monoidal category such that the canonical functor $X \to X/{\sim}$ is strict monoidal. If $X$ is braided (symmetric),
then the category $X/{\sim}$ and the functor $X \to X/{\sim}$ are braided (resp., symmetric) too.

\begin{example}
If $A$ is an $\Omega$-magma, then $\Id(A)$ is a monoidal congruence on $\mathcal M (\Omega)$.
\end{example}

Given a set $V$ of $\Omega$-polynomial identities, consider the minimal monoidal congruence $\sim$ on $\mathcal M(\Omega)$
such that $f\sim g$ for every $f=g$ from $V$.
We call $\mathcal M_V(\Omega) := \mathcal M(\Omega)/{\sim}$ the category of \textit{$V$-relative
$\Omega$-monomials}. Then $\mathcal M_V(\Omega)$ is a BMAT too and an $\Omega$-magma $A$ belongs to $\mathop\mathrm{Var}(V)$
if and only if there exists a braided strong monoidal functor $F \colon \mathcal M_V(\Omega) \to \mathcal C$
such that $F(1)=A$ and $F(\omega)=\omega_A$ for all $\omega \in \Omega$.

(Co)restricting the functors from Proposition~\ref{PropositionOmegaMagmaBSMFunctorsEquivalence}, we get

\begin{proposition}\label{PropositionVarBSMFunctorsEquivalence} Let $\Omega$ be a set together with maps $s,t \colon \Omega \to \mathbb Z_+$, let $V$ be a set of $\Omega$-polynomial identities
	and let $\mathcal C$ be a braided monoidal category. Then the categories $\mathrm{Var}(V)$ and $\mathbf{BSMF}(\mathcal M_V(\Omega), \mathcal C)$ are equivalent.
\end{proposition}

In other words, to every variety $\mathop\mathrm{Var}(V)$ of $\Omega$-magmas we assign its  
BMAT~$\mathcal M_V(\Omega)$.
 As in the case of ordinary algebraic theories, in this way we can obtain every BMAT:

\begin{proposition}\label{PropositionBMATisMVOmega}
	An arbitrary BMAT~$\mathbb{A}$ is braided monoidally isomorphic to
	$\mathcal M_V(\Omega)$ for some $V$ and $\Omega$.
\end{proposition}
\begin{proof}
Denote by $\Omega$ the set of all morphisms $\omega$ in~$\mathbb{A}$ and define the maps $s,t \colon \Omega \to \mathbb R_+$
by $s(\omega):=m$, $s(\omega):=n$ for $\omega \in \mathbb{A}(m,n)$.
Then $1$ is an $\Omega$-magma in~$\mathbb{A}$ and there exists a unique strict braided monoidal functor $F \colon \mathcal M(\Omega) \to \mathbb{A}$ such that $F(1)=1$ and $F(\omega)=\omega$ for all $\omega \in \Omega$. Note that $F$
is surjective on hom-sets. Denote by $V$ the set of polynomial identities resulting from the kernel congruence $\sim$ of $F$. Then $ \mathbb{A} \cong \mathcal M(\Omega)/{\sim} = \mathcal M_V(\Omega)$.
\end{proof}

As we have already mentioned,
if an $\Omega$-magma $A$ belongs to $\mathop\mathrm{Var}(V)$ for some set $V$ of polynomial identities,
it is natural to use in the definition of additional polynomial identities that $A$ may satisfy,
the category $\mathcal M_V(\Omega)$ instead of $\mathcal M(\Omega)$.
For $A$ belonging to $\mathop\mathrm{Var}(V)$
the functor $\mathcal E_A$ factors through $\mathcal M_V(\Omega)$.
Denote the corresponding functor $\mathcal M_V(\Omega) \to \mathcal C$
again by $\mathcal E_A$. For $V$-relative
$\Omega$-monomials $f,g$
we say that  $f\equiv g$ is a \textit{polynomial identity in $A$} or that $f\equiv g$ \textit{holds in $A$} if  $\mathcal E_A(f)=\mathcal E_A(g)$.

\begin{definition} Let $V$ be a set of $\Omega$-polynomial identities.
	We say that a polynomial identity $f\equiv g$ \textit{follows} from $V$
	for some $\Omega$-monomials $f$ and $g$ if the images of $f$ and $g$ in $\mathcal M_V(\Omega)$ coincide.
	Let $A$ be an $\Omega$-magma. We say that $V$ \textit{generates}
	polynomial identities in $A$ or that $V$ is a \textit{basis} for polynomial identities in $A$ if the polynomial identities from $V$ hold in $A$ and the functor $\mathcal E_A \colon \mathcal M_V(\Omega) \to \mathcal C$ is faithful. In other words, $V$ generates polynomial identities in $A$ if any polynomial identity $f \equiv g$ in $A$ follows from $V$.
\end{definition}	

\begin{remark}\label{RemarkTau11}
Recall that if $\tau_{C,A}\tau_{A,C}=\id_{A\otimes C}$ and $\tau_{C,B}\tau_{B,C}=\id_{B\otimes C}$
for some objects $A,B,C$ in a braided monoidal category $\mathcal C$ with a braiding $\tau$, then
$\tau_{C,A\otimes B}\tau_{A\otimes B, C}=\id_{(A\otimes B) \otimes C}$ too.
Now the induction argument implies that if the polynomial identity $\tau_{1,1}^2 \equiv \id_2$ belongs to $V$,
then the category $\mathcal M_V(\Omega)$ is symmetric.
\end{remark}

\begin{remark} If $A$ is an ordinary algebraic structure of signature $\Omega$, i.e. $t(\omega)=1$ for all $\omega \in \Omega$ and $\mathcal C = \mathbf{Sets}$,
	then $\Omega$-polynomial identities defined above will  all be ``multilinear'', i.e. every variable will appear in each side exactly once. As we have mentioned at the end of Section~\ref{SubsectionAlgebraicStructures}, 
	in order to get all polynomial identities in the sense of Section~\ref{SubsectionAlgebraicStructures},
	one has to add to $\Omega$ two new symbols, $\varepsilon$ and $\Delta$,
	where $s(\Delta)=s(\varepsilon)=1$, $t(\Delta)=2$, $t(\varepsilon)=0$,
	define $\varepsilon_A(a)=*$, $\Delta_A a := (a,a)$ for all $a\in A$
	and consider $\Omega \sqcup \lbrace\Delta, \varepsilon \rbrace$-polynomial identities.
	Let $V$ be the set consisting of the following polynomial identities:
		\begin{enumerate}
		\item $\tau_{1,1}^2\equiv \id_2$, forcing the braiding $\tau_{k,\ell} \colon (k+\ell) \to (k+\ell)$, $k,\ell \in \mathbb Z_+$, to be symmetric (see Remark~\ref{RemarkTau11});
		\item $(\Delta \otimes \id_1)\Delta \equiv (\id_1 \otimes \Delta )\Delta$, $(\varepsilon \otimes \id_1)\Delta = (\id_1 \otimes \varepsilon)\Delta \equiv \id_1$, forcing $(1, \Delta, \varepsilon)$ to be a comonoid;
		\item $\varepsilon \omega \equiv \varepsilon^{{}\otimes s(\omega)}$
	and $\Delta \omega \equiv (\omega\otimes\omega) \tau\Delta^{{}\otimes s(\omega)}$
	for all $\omega \in \Omega$
	where $\tau$ is a composition of swaps $\tau_{k,\ell}$ that corresponds
	to the permutation $\left(\begin{smallmatrix} 1 & 2 & \ldots & s(\omega) & s(\omega)+1 &  s(\omega)+3 & \ldots  &2s(\omega) \\ 
	1 & 3 & \ldots & 2s(\omega)-1 & 2 & 4  & \ldots & 2s(\omega)
\end{smallmatrix}\right)$,
		making it possible to swap $\varepsilon$ and $\Delta$ with $\omega$ in the same way as in 
		ordinary algebraic structures of signature $\Omega$.
	\end{enumerate}
	Then every $f \in \mathcal M_V(\Omega \sqcup \lbrace\Delta, \varepsilon \rbrace)(n,1)$
	equals $\alpha\beta\gamma$ where $\alpha$ is a composition of $\omega \in \Omega$, $\beta$ is a composition of swaps $\tau_{k,\ell}$ and $\gamma$ is a composition of $\Delta$ and $\varepsilon$ where $\omega$, $\tau_{k,\ell}$, $\Delta$ and $\varepsilon$ can appear monoidally multiplied by $\id_r$ for some $r\in\mathbb N$.
	Substituting the generators $x_1,\ldots, x_n$ of the free algebraic structure $\mathcal F_\Omega (x_1, \ldots, x_n)$
	for the corresponding arguments of $f$, we see that the set $\mathcal M_V(\Omega \sqcup \lbrace\Delta, \varepsilon \rbrace)(n,1)$ can be identified with $\mathcal F_\Omega (x_1, \ldots, x_n)$ for every $n\in\mathbb N$.
	Under this identification, polynomial identities of $A$ correspond to polynomial identities of~$A$.
\end{remark}

\section{Polynomial identities in $\Omega$-magmas in linear categories}

\subsection{Polynomial identities in algebras over a field}
Again, before making a generalization, we recall the classical definition of polynomial identities and their codimensions in associative algebras~\cite{DrenKurs,GiaZai}.

Let $\mathbbm{k}$ be a field. Denote by $\mathbbm k \langle X \rangle$ the free non-unital associative algebra on
the countable set $X=\lbrace x_1, x_2, \ldots \rbrace$, i.e. the algebra of polynomials without a constant term with coefficients from $\mathbbm k$ in the non-commuting variables from $X$. Let $A$ be an associative $\mathbbm{k}$-algebra.
We say that $f\in \mathbbm k \langle X \rangle$ is a \textit{polynomial identity} in $A$ and write $f\equiv 0$
if $f(a_1, \ldots, a_n)=0$ for all $a_i \in A$ where $n\in\mathbb N$ is the number of variables that appear in $f$.
In other words, $f\equiv 0$ if and only if $\varphi(f)=0$ for every algebra homomorphism $\varphi \colon \mathbbm k \langle X \rangle \to A$. The set $\Id(A)$ forms an ideal in $\mathbbm k \langle X \rangle$ that is invariant under all endomorphisms of $\mathbbm k \langle X \rangle$.

Consider the vector space $$P_n = \left\lbrace  \left.\sum\limits_{\sigma \in S_n} \alpha_\sigma\, 
x_{\sigma(1)} x_{\sigma(2)} \ldots x_{\sigma(n)} \right| \alpha_\sigma \in \mathbbm k \right\rbrace \subset \mathbbm k \langle X \rangle$$ of multilinear polynomials in $x_1, \ldots, x_n$.
(Here $S_n$ is the $n$th symmetric group.)

The number $c_n(A) := \dim \frac{P_n}{P_n \cap \Id(A)}$, $n\in\mathbb N$, is called the $n$th \textit{codimension}
of polynomial identities of $A$. Let  $\Hom_{\mathbbm k}(A^{{}\otimes n},A)$ be the space of all linear maps $A^{{}\otimes n} \to A$. Then $c_n(A)$ is just the dimension of the subspace of $\Hom_{\mathbbm k}(A^{{}\otimes n},A)$ consisting of all the maps
that can be realized using the multiplication in $A$. A detailed study of the asymptotic behavior of the codimension sequence can be found in~\cite{GiaZai}.

Analogous definitions can be made for not necessarily associative algebras too. Instead of $\mathbbm k \langle X \rangle$ one must use the absolutely free non-associative algebra $\mathbbm k \lbrace X \rbrace$ and the corresponding non-associative multilinear polynomials. The alternative definition of $c_n(A)$ as the dimension of the corresponding subspace in $\Hom_{\mathbbm k}(A^{{}\otimes n},A)$ implies that for an associative algebra $A$ the numbers $c_n(A)$ do not depend on whether we use $\mathbbm k \langle X \rangle$ or $\mathbbm k \lbrace X \rbrace$ in their definition.

\subsection{$\Omega$-polynomials with coefficients in a ring}

Let $R$ be a unital commutative associative ring. (Below all commutative rings will be associative too and the word ``associative'' will usually be omitted.) Recall that a category $\mathcal C$ is \textit{$R$-linear} if it is enriched over
the category ${}_R \mathsf{Mod}$ of $R$-modules, i.e. if all hom-sets of $\mathcal C$ are $R$-modules
and the composition of morphisms is $R$-bilinear. For $R$-linear monoidal categories we require that the monoidal product is $R$-bilinear too. A functor $F \colon \mathcal C \to \mathcal D$ between $R$-linear categories $\mathcal C$ and $\mathcal D$
is called \textit{$R$-linear} if $F$ is an $R$-linear map on hom-sets.

For $R$-linear braided monoidal categories 
$\mathcal C$ and $\mathcal D$ denote by $R\text{-}\mathbf{LBSMF}(\mathcal C,\mathcal D)$
the category of $R$-linear braided strong monoidal functors $\mathcal C \to \mathcal D$ where morphisms between $F,G \colon \mathcal C \to \mathcal D$ are all monoidal transformations $\alpha \colon F \Rightarrow G$.

For $\Omega$-magmas in $R$-linear braided monoidal categories one can introduce polynomial identities with coefficients in $R$.

Consider the braided monoidal category
$\mathcal P(R, \Omega)$ where the objects, the monoidal product on them and the braiding are the same as in $\mathcal M(\Omega)$
and for every $m,n\in\mathbb{Z}_+$ the hom-set
$\mathcal P(R, \Omega)(m,n)$ is just the free $R$-module with the basis $\mathcal M(\Omega)(m,n)$.
The compositions and the monoidal product on morphisms are extended from $\mathcal M(\Omega)$ to
$\mathcal P(R, \Omega)$ by the $R$-linearity.

Then  $\mathcal P(R, \Omega)$ is an $R$-linear BMAT, which we call the \textit{category of $\Omega$-polynomials with coefficients in $R$}. Morphisms in $\mathcal P(R, \Omega)$ are called \textit{$\Omega$-polynomials with coefficients in $R$}. 

 \begin{remark}
	Denote by $R\text{-}\mathbf{LBMAT}$ the category of $R$-linear BMATs where the morphisms are $R$-linear braided strict monoidal functors and by $U \colon R\text{-}\mathbf{LBMAT} \to ((\mathbb Z_+ \times \mathbb Z_+) \downarrow \mathbf{Sets})$ the functor
	that assigns every $R$-linear BMAT $\mathbb A$ the function $(m,n) \mapsto \mathbb A(m,n)$, $m,n\in \mathbb Z_+$.
	Then we again get a bijection
	$$R\text{-}\mathbf{LBMAT}(\mathcal P(R, \Omega), \mathbb A) \mathbin{\widetilde\to} ((\mathbb Z_+ \times \mathbb Z_+) \downarrow \mathbf{Sets})(\Omega, U \mathbb A)$$
	natural in $\Omega$ and $\mathbb A$. Hence the functor $\mathcal P(R, -)$ is the free functor left adjoint to the forgetful functor $U$. The functor $\mathcal P(R, -)$ first appeared in~\cite{MarklPROPDef} where BMATs $\mathcal P(R, \Omega)$ were called free theories.
\end{remark}

For every $\Omega$-magma $A$ 
in an $R$-linear braided monoidal category~$\mathcal C$  extend
the functor $\mathcal E_A \colon \mathcal M(\Omega) \to \mathcal C$ defined above to the functor $\mathcal P(R, \Omega) \to \mathcal C$ by linearity. Namely, if $k,m,n \in \mathbb Z_+$ and $\alpha_i \in R$, $f_i \in \mathcal M(\Omega)(m,n)$ for $1\leqslant i \leqslant k$, let $$\mathcal E_A\left(\sum_{i=1}^k \alpha_i f_i\right):=\sum_{i=1}^s \alpha_i \mathcal E_A(f_i).$$

An $\Omega$-polynomial $f$ is a \textit{polynomial identity in $A$ with coefficients in $R$}  if $\mathcal E_A(f) = 0$. In this case we say that $A$ satisfies $f\equiv 0$. Polynomial identities in $A$ with coefficients in $R$ are morphisms in the semicategory\footnote{The definition of a semicategory is identical to that of a category, except that the existence of identity morphisms is not required.} $\Id(A,R)$ where the set of objects is $\mathbb Z_+$ and $\Id(A,R)(m,n) := \Ker \left(\mathcal E_A \bigr|_{\mathcal P(R, \Omega)(m,n)}\right)$
for every $m,n\in\mathbb Z_+$.

Let $V$ be a set of $\Omega$-polynomials with coefficients in $R$. The full subcategory $\mathop\mathrm{Var}(V)$ of $\Omega\text{-}\mathsf{Magma}(\mathcal C)$, consisting of all $\Omega$-magmas satisfying $f\equiv 0$ for all $f\in V$, is again called a \textit{variety} of $\Omega$-magmas. 

Recall that a \textit{monoidal ideal} $I$ in an $R$-linear monoidal category $X$ is
a system of $R$-submodules $I(a,b) \subseteq X(a,b)$, where $a,b$ are objects in $X$, such that
\begin{enumerate}
	\item $gf \in I(a,c)$ for all $f\in X(a,b)$, $g\in X(b,c)$ where either $f\in I(a,b)$
or $g\in I(b,c)$;
\item $f \otimes g \in I(a \otimes c, b \otimes d)$ for all $f\in X(a,b)$, $g\in X(c,d)$ where either $f\in I(a,b)$
or $g\in I(c,d)$
\end{enumerate}
for all objects $a,b,c,d$ in $X$. Then the \textit{factor category} $X/I$ is the category with the same objects as in $X$
and the hom-sets $X(a,b)/I(a,b)$. It is easy to see that $X/I$ is again an $R$-linear monoidal category and the canonical functor $X \to X/I$ is an $R$-linear strict monoidal functor. If $X$ is braided (symmetric), then the category $X/I$ and the functor $X \to X/I$ are braided (resp., symmetric) too.

\begin{example}
For every $\Omega$-magma $A$ the hom-sets $\Id(A,R)(m,n)$ of $\Id(A,R)$ form a monoidal ideal in $\mathcal P(R, \Omega)$.
\end{example}

Again, given a set $V$ of $\Omega$-polynomials with coefficients in $R$, we consider the minimal monoidal ideal $I$
 in $\mathcal P(R, \Omega)$ such that $V \subseteq \bigcup\limits_{m,n\in\mathbb Z_+} I(m,n)$.
We call the category $\mathcal P_V(R,\Omega) := \mathcal P(R, \Omega)/I$ the category of \textit{$V$-relative
	$\Omega$-polynomials with coefficients in $R$}. Then an $\Omega$-magma $A$ belongs to $\mathop\mathrm{Var}(V)$
if and only if there exists a $R$-linear braided strong monoidal functor $F \colon \mathcal P_V(R,\Omega) \to \mathcal C$
such that $F(1)=A$ and $F(\omega)=\omega_A$ for all $\omega \in \Omega$. The category
$\mathcal P_V(R,\Omega)$ is an $R$-linear BMAT and the analogs of Propositions~\ref{PropositionOmegaMagmaBSMFunctorsEquivalence}, \ref{PropositionVarBSMFunctorsEquivalence} and~\ref{PropositionBMATisMVOmega} hold too:
\begin{proposition}\label{PropositionOmegaMagmaVarRLBSMFunctorsEquivalence} Let $\Omega$ be a set together with maps $s,t \colon \Omega \to \mathbb Z_+$
	and let $\mathcal C$ be a $R$-linear braided monoidal category for a unital commutative ring $R$. Then the categories $\Omega\text{-}\mathsf{Magma}(\mathcal C)$ and $R\text{-}\mathbf{LBSMF}(\mathcal P(R,\Omega), \mathcal C)$ are equivalent. Moreover, if $V$ is a set of $\Omega$-polynomials with coefficients in $R$,
	then the categories  $\mathrm{Var}(V)$ and $R\text{-}\mathbf{LBSMF}(\mathcal P_V(R,\Omega), \mathcal C)$ are equivalent.
\end{proposition}
\begin{proof}
	Again, consider the correspondence $A \mapsto \mathcal E_A$ and $F \mapsto F(1)$ where $A$ is an $\Omega$-magma and
	 $F \colon \mathcal P(R,\Omega) \to \mathcal C$ is an $R$-linear braided strong monoidal functor.
\end{proof}	
\begin{proposition}\label{PropositionRLBMATisPRVOmega} Let $R$ be a unital commutative ring.
	An arbitrary $R$-linear BMAT~$\mathbb{A}$ is linearly braided monoidally isomorphic to
	$\mathcal P_V(R, \Omega)$ for some $V$ and $\Omega$.
\end{proposition}
\begin{proof} Denote by $\Omega$ the union of generating sets for all hom-$R$-modules of $\mathbb{A}$ and define the maps $s,t \colon \Omega \to \mathbb R_+$
	by $s(\omega):=m$, $s(\omega):=n$ for $\omega \in \mathbb{A}(m,n)$.
	Then $1$ is an $\Omega$-magma in~$\mathbb{A}$ and there exists a unique $R$-linear strict braided monoidal functor $F \colon \mathcal P_V(R,\Omega) \to \mathbb{A}$ such that $F(1)=1$ and $F(\omega)=\omega$ for all $\omega \in \Omega$. Again, $F$ is surjective on hom-sets. Denote by $V$ the set of polynomial identities resulting from the kernel ideal $I$ of $F$. Then $ \mathbb{A} \cong \mathcal P(R,\Omega)/I = \mathcal P_V(R,\Omega)$.
\end{proof}	

Again, if an $\Omega$-magma $A$ belongs to $\mathop\mathrm{Var}(V)$ for some set $V$ of $\Omega$-polynomials with coefficients in $R$, it is natural to define additional polynomial identities of~$A$ using $V$-relative
$\Omega$-polynomials with coefficients in $R$. The functor $\mathcal P_V(R,\Omega) \to \mathcal C$
induced by $\mathcal E_A$ is denoted again by $\mathcal E_A$.

\begin{definition} Let $V$ be a set of $\Omega$-polynomials with coefficients in $R$.
	We say that a polynomial identity $f\equiv 0$, where $f$ is an $\Omega$-polynomial, \textit{follows} from $V$
	if the image of $f$ in $\mathcal P_V(R, \Omega)$ is zero.
	Let $A$ be an $\Omega$-magma.  We say that $V$ \textit{generates}
	polynomial identities in $A$  with coefficients in $R$ (or that $V$ is a \textit{basis} for them) if $A$
	belongs to $\mathop\mathrm{Var}(V)$
	 and the functor $\mathcal E_A \colon \mathcal P_V(R, \Omega) \to \mathcal C$ is faithful. In other words, $V$ generates
	polynomial identities of~$A$ if any polynomial identity $f \equiv 0$ holding in $A$ follows from~$V$.
\end{definition}	

\begin{example}\label{ExampleAssocAlgebra}
	Let $\Omega = \lbrace \mu \rbrace$, $s(\mu)=2$, $t(\mu)=1$, $V_\mathrm{assoc} := \lbrace \mu(\mu \otimes \id_1) -  \mu( \id_1 \otimes \mu )\rbrace \cup V_\mathrm{symm} $ where $V_\mathrm{symm} :=\lbrace \tau_{1,1}^2-\id_2\rbrace$ and let $\mathbbm k$ be a field.
	In other words, $V_\mathrm{assoc}$ consists of the associator $(xy)z-x(yz)$ and the identities that force the braiding $\tau_{k,\ell} \colon (k+\ell) \to (k+\ell)$ to be symmetric (see Remark~\ref{RemarkTau11}). Then $\mathcal P_{V_\mathrm{assoc}}(\mathbbm k, \Omega)(n,1)$, where $n\in\mathbb N$, coincides with the vector space $P_n$ of associative multilinear polynomials in the variables $x_1, \ldots, x_n$. Under this identification,
	polynomial identities in an arbitrary associative algebra $A$ correspond to polynomial identities in $A$.
\end{example}

\subsection{Codimensions of $\Omega$-polynomial identities}

Consider an $\Omega$-magma $A$ in a $\mathbbm{k}$-linear braided monoidal category $\mathcal C$
where $\mathbbm{k}$ is a field.  Let $m,n \in \mathbb{Z}_+$.  Then $$c_{m,n}(A) := \dim \mathcal E_A\bigl(\mathcal P(\mathbbm{k}, \Omega)(m, n) \bigr)$$ is called the \textit{$(m,n)$-codimension} of polynomial identities of~$A$.

\begin{proposition}
Let $A$ be an $\Omega$-algebra over a field $\mathbbm{k}$, i.e. an $\Omega$-magma in $\mathbf{Vect}_\mathbbm{k}$. 
If $\dim A < +\infty$, then $c_{m,n}(A) \leqslant (\dim A)^{m+n}$ for every $m,n \in \mathbb{Z}_+$.
\end{proposition}
\begin{proof}
The space $\mathcal E_A\bigl(\mathcal P(\mathbbm{k}, \Omega)(m, n)\bigr)$ is the subspace of $\Hom_{\mathbbm k}(A^{\otimes m},  A^{\otimes n}) $ that consists of all linear maps $A^{\otimes m} \to A^{\otimes n}$ that can be constructed using the operations from $\Omega$ together with permutations of variables.
Now the upper bound on $c_{m,n}(A)$ follows from the equality $\dim \Hom_{\mathbbm k}(A^{\otimes m},  A^{\otimes n}) = (\dim A)^{m+n}$.
\end{proof}	

\begin{example} If $A$ is an ordinary algebra over a field $\mathbbm{k}$, then $c_{n,1}(A)$ is the $n$th ordinary codimension $c_n(A)$ of polynomial identities of $A$ for $n\in \mathbb N$.
\end{example}	

\subsection{$\Omega^*$-magmas}

For a set $\Omega$ together with maps $s,t \colon \Omega \to \mathbb Z_+$
define the set $\Omega^* := \lbrace \omega^* \mid \omega \in \Omega \rbrace$
and the maps $s,t \colon \Omega^* \to \mathbb Z_+$ by $s(\omega^*) := t(\omega)$
and $t(\omega^*) := s(\omega)$ for all $\omega \in \Omega$. We call the signature $\Omega^*$ \textit{dual}
to $\Omega$. By the definition, $\Omega^{**} := \Omega$.

\begin{example}
	Unital algebras are $\Omega$-magmas for $\Omega$ from Example~\ref{ExampleUnitalAlgebraOmega}
	while coalgebras are $\Omega^*$-magmas for the same $\Omega$.
\end{example}

The map $(-)^* \colon \Omega \to \Omega^*$ induces the contravariant functor $(-)^* \colon \mathcal M(\Omega) \to \mathcal M(\Omega^*)$ where $m^*:=m$ for objects, $(u \otimes v)^* := u^* \otimes v^*$  for morphisms, $\tau_{m,n}^*:= \tau^{-1}_{m^*,n^*}$. For every unital commutative ring $R$
the functor $(-)^*$ extends uniquely to an $R$-linear contravariant functor $(-)^* \colon \mathcal P(R,\Omega) \to \mathcal P(R,\Omega^*)$

Consider finite dimensional $\Omega$-algebras $A$ over a field $\mathbbm{k}$, which can be treated
as $\Omega$-magmas in the category $\mathbf{Vect}_\mathbbm{k}^\mathrm{f.d.}$
of finite dimensional vector spaces over $\mathbbm{k}$. After the natural identifications $\left(A^{\otimes n}\right)^* \cong \left(A^*\right)^{\otimes n}$, the space $A^*$
of linear functions $A\to \mathbbm{k}$ becomes an $\Omega^*$-algebra with $\omega_{A^*} := \omega_A^*$ for all $\omega \in \Omega$. The commutativity of the diagram below, where both horizontal arrows are bijections, implies that $f^*$ is a polynomial identity in $A^*$ if and only if $f$ is a polynomial identity in $A$ and $c_{m,n}(A)=c_{n,m}(A^*)$ for all $m,n \in \mathbb{Z}_+$:
$$\xymatrix{ \mathcal P(\mathbbm{k},\Omega)\left(m,n\right) \ar[rr]^{(-)^*} \ar[d]^{\mathcal E_A} & & \mathcal P(\mathbbm{k},\Omega^*)\left(n,m\right) \ar[d]^{\mathcal E_{A^*}}\\
\mathbf{Vect}_\mathbbm{k}^\mathrm{f.d.}\left(A^{\otimes m}, A^{\otimes n}\right) \ar[rr]^{(-)^*} & & \mathbf{Vect}_\mathbbm{k}^\mathrm{f.d.}\left((A^*)^{\otimes n}, (A^*)^{\otimes m}\right)
}
$$
  
 \section{Polynomial identities in graded and (co)module algebras}\label{SectionPIGr(Co)mod}
 
 \subsection{Polynomial $U$-identities}
 Motivated by applications of the reconstruction technique to $H$-(co)module algebras (see e.g. \cite[Chapter 5]{EGNObook}),
 we give the following definitions.
 
  Let $U \colon \mathcal C \to \mathcal D$ be a strong monoidal functor between a monoidal category $\mathcal C$ and a braided monoidal category $\mathcal D$.
Consider the set-theoretical monoid $\End(U)$ of (not necessarily monoidal) natural transformations $U \Rightarrow U$.
For every $\Omega$-magma $A$ in $\mathcal C$ the object $UA$ admits the structure of an $\Omega \sqcup \End(U)$-magma in $\mathcal D$
where for $h\in \End(U)$ we have $s(h):=t(h):=1$ and the operation $h_{UA} \colon UA \to UA$ is just the component $h_A$ of the natural transformation $h$.
Then the \textit{polynomial $U$-identities} of the $\Omega$-magma $A$ in $\mathcal C$ are the polynomial identities of the $\Omega \sqcup \End(U)$-magma $A$ in $\mathcal D$. If $\mathcal D$ is linear over a field $\mathbbm{k}$, one can define the codimensions
$c_{m,n}^U(A)$
of polynomial $U$-identities too.
Let $m,n \in \mathbb{Z}_+$.  Then $$c_{m,n}^U(A) := \dim \mathcal E_{UA}\bigl(\mathcal P(\mathbbm{k}, \Omega \sqcup \End(U))(m, n) \bigr)$$ is called the \textit{$(m,n)$-codimension of polynomial $U$-identities} of~$A$.

\subsection{Polynomial $H$-identities in $H$-module algebras}\label{SubsectionPIHmod}

In order to show why multilinear polynomial $H$-identities in $H$-module algebras and their codimensions are indeed a particular case of polynomial $U$-identities and their codimensions introduced above, we recall the corresponding definitions~\cite{BahturinLinchenko, BereleHopf, ASGordienko15}.  (We refer the reader to~\cite{DNR, Montgomery, SweedlerBook}
for an account of Hopf algebras and algebras with Hopf algebra actions.)

The free associative non-unital algebra $\mathbbm k \langle X \rangle$ over a field $\mathbbm k$ admits a $\mathbb Z$-grading  $\mathbbm k \langle X \rangle = \bigoplus\limits_{n=1}^\infty 
\mathbbm k \langle X \rangle^{(n)}$ where $\mathbbm k \langle X \rangle^{(n)}$ is the linear span of all monomials of total degree $n$.

Let $H$ be a Hopf algebra over the field $\mathbbm k$. Recall that the category ${}_H\mathsf{Mod}$ of left $H$-modules is monoidal where the monoidal product of $H$-modules $M$ and $N$ is their tensor product over the base field $\mathbbm k$
and $h(m\otimes n) := h_{(1)}m \otimes h_{(2)}n$ for all $h\in H$, $m\in M$, $n\in N$.
(Here we use Sweedler's notation $\Delta h=h_{(1)}\otimes h_{(2)}$ for the comultiplication $\Delta \colon H \to H \otimes H$
and $h\in H$, where the sign of sum is omitted.) An algebra $A$ over a field $\mathbbm k$ 
is an \textit{$H$-module algebra} if $A$ is a (left) $H$-module and the multiplication $A \otimes A \to A$
is an $H$-module homomorphism. 

Consider the algebra $$\mathbbm k \langle X  | H \rangle := \bigoplus\limits_{n=1}^\infty H^{{}\otimes n} \otimes
\mathbbm k \langle X \rangle^{(n)} $$ with the multiplication $(u \otimes v)(t \otimes w):= u \otimes t \otimes vw$
for
$u\in H^{{}\otimes m}$, $t\in H^{{}\otimes n}$, $v\in \mathbbm k \langle X \rangle^{(m)}$, $w\in \mathbbm k \langle X \rangle^{(n)}$, $m,n \in \mathbb N$. Then $\mathbbm k \langle X  | H \rangle$ is a left $H$-module algebra where $h(u \otimes v)=hu \otimes v$ for all $u\in H^{{}\otimes m}$, $v\in \mathbbm k \langle X \rangle^{(m)}$, $m \in \mathbb N$.
The subset $\lbrace  1 \otimes x_i \mid i\in \mathbb N\rbrace$ can be identified with $X$.
The algebra $\mathbbm k \langle X  | H \rangle$ is called the \textit{free $H$-module algebra}.
Note that $\mathbbm k \langle X  | H \rangle$ is just the (non-unital) tensor algebra of the free $H$-module generated by the set $X$.
Elements of $\mathbbm k \langle X  | H \rangle$ are called \textit{$H$-polynomials}.
 Every map $X \to A$, where $A$ is an associative $H$-module algebra, extends to a homomorphism $\mathbbm k \langle X  | H \rangle \to A$ of algebras and $H$-modules in a unique way. An $H$-polynomial $f\in \mathbbm k \langle X  | H \rangle$
 is called a \textit{polynomial $H$-identity} of $A$ if $\varphi(f)=0$ for all homomorphisms $\varphi \colon \mathbbm k \langle X  | H \rangle \to A$ of algebras and $H$-modules. The set $\Id^H(A)$ of polynomial $H$-identities of 
 $A$ is an $H$-invariant ideal of $ k \langle X  | H \rangle$.

 Let $x_{i_1}^{h_1} \ldots x_{i_n}^{h_n} := h_1 \otimes h_2 \otimes \dots \otimes h_n \otimes x_{i_1} \ldots x_{i_n}$
 for $h_1, \ldots, h_n \in H$, $i_1, \ldots, i_n, n \in\mathbb N$.
 The subspace $$P_n^H = \left\langle\left. x_{\sigma(1)}^{h_1} \ldots  x_{\sigma(n)}^{h_n} \right| \sigma\in S_n,\ h_i \in H \right\rangle_{\mathbbm k} \subset \mathbbm k \langle X  | H \rangle$$
 is called the space of \textit{multilinear $H$-polynomials} in variables $x_1, \ldots, x_n$.
 The dimension $c_n^H(A) := \dim\frac{P_n^H }{ P_n^H  \cap\, \Id^H(A)}$ is called the \textit{$n$th codimension
 of polynomial $H$-identities} of  $A$.

	Let $\mathcal C := {}_H\mathsf{Mod}$,	$\mathcal D := \mathbf{Vect}_{\mathbbm k}$ and let $U \colon \mathcal C \to \mathcal D$ be the corresponding
	forgetful functor.
	For an arbitrary natural transformation $\theta \colon U \Rightarrow U$ 
	define $h_\theta := \theta_H(1_H)$ where $H$ is considered to be a left module over itself. Let $m\in M$ where $M$ is a left $H$-module. Define the $H$-module homomorphism
	$f \colon H \to M$ by $f(h):= hm$, $h\in H$. The diagram below is commutative by the naturality of $\theta$:
	$$\xymatrix{ H \ar[r]^{\theta_H} \ar[d]_f & H \ar[d]_f \\
	M \ar[r]^{\theta_M} & M}$$
Hence $\theta_M(m)=\theta_M f(1_H)= f \theta_M (1_H) = f(h_\theta)=h_\theta m$.
Conversely, every element $h\in H$ defines a natural transformation $\theta \colon U \Rightarrow U$ 
by $\theta_M(m):=hm$ for $m\in M$ where $M$ is a left $H$-module.
Therefore, $\End(U)$ can be identified with $H$.

Now consider an $H$-module algebra $A$ and let $\Omega=\lbrace \mu \rbrace$, $s(\mu)=2$, $t(\mu)=1$.
Every element of $\mathcal P(\mathbbm{k}, \lbrace \mu \rbrace \sqcup \End(U))(n, 1)=\mathcal P(\mathbbm{k}, \lbrace \mu \rbrace \sqcup H)(n, 1)$, where $n\in \mathbb N$,
corresponds to a linear map $A^{{}\otimes n} \to A$ that can be represented by an $H$-polynomial.
Therefore, the codimension
$\dim \mathcal E_{UA}\bigl(\mathcal P(\mathbbm{k}, \lbrace \mu \rbrace \sqcup H)(n,1) \bigr)$
of polynomial $U$-identities
equals $c_n^H(A)$.
In contrast with traditional polynomial $H$-identities, where operators from $H$ are applied only to variables
themselves, not their products, in the elements of $\mathcal P(\mathbbm{k}, \lbrace \mu \rbrace \sqcup H)(m, n)$ operators from $H$ can be applied in any place. However if we let $V := V_\mathrm{assoc} \cup \lbrace h(xy)-(h_{(1)} x)(h_{(2)} y) \mid h\in H\rbrace$,
then $\mathcal P_V(\mathbbm{k}, \lbrace \mu \rbrace \sqcup H)(n, 1)$ can be identified with $P_n^H$.
Under this identification, polynomial identities correspond to polynomial identities.

In addition, one can consider a more general situation.
Recall that an algebra $A$ is an \textit{algebra with a generalized $H$-action} if
 $A$ is a left $H$-module for an associative unital algebra $H$
and  for every $h \in H$ there exist $k\in \mathbb N$ and $h'_i, h''_i, h'''_i, h''''_i \in H$ where $1\leqslant i \leqslant k$
such that
\begin{equation*}
h(ab)=\sum_{i=1}^k\bigl((h'_i a)(h''_i b) + (h'''_i b)(h''''_i a)\bigr) \text{ for all } a,b \in A.
\end{equation*}
Multilinear $H$-polynomials are introduced in the same way as for $H$-module algebras~\cite{BereleHopf,ASGordienko15}.
Again, multilinear $H$-polynomials of degree $n$ can be identified with elements of $\mathcal P_V(\mathbbm{k}, \lbrace \mu \rbrace \sqcup H)(n, 1)$ for $$V = V_\mathrm{assoc} \cup \left\lbrace\left. h(xy)-\sum\limits_{i=1}^k\bigl((h'_i x)(h''_i y) + (h'''_i y)(h''''_i x)\bigr) \right| h\in H\right\rbrace.$$
As before, polynomial identities correspond to polynomial identities.

\subsection{Graded polynomial identities and polynomial $H$-identities in $H$-comodule algebras}
Here we first recall the definition of polynomial $H$-identities in $H$-comodule algebras introduced in~\cite{AljaKassel}
as a generalization of graded polynomial identities (see e.g.~\cite[Section 10.5]{GiaZai}).

Let $H$ be again a Hopf algebra $H$ over a field $\mathbbm k$.
Then the category $\mathsf{Comod}^H$ of right $H$-comodules is monoidal
where for $H$-comodules $M$ and $N$ the comodule map $\rho_{M\otimes N} \colon M \otimes N \to M \otimes N \otimes H$
is defined by $\rho_{M\otimes N}(m \otimes n)=m_{(0)} \otimes n_{(0)} \otimes m_{(1)} n_{(1)}$
for all $m\in M$ and $n\in N$. (Here we use Sweedler's notation $\rho_M(m)=m_{(0)}\otimes m_{(1)}$
for $m\in M$ and  the comodule map $\rho_M \colon M \to M \otimes H$.) Below will denote the structure maps
for all right comodules by the same letter $\rho$.

Let $A$ be an \textit{$H$-comodule algebra}, i.e. a $\mathbbm k$-algebra that is a (right) $H$-comodule
where the multiplication $A \otimes A \to A$ is an $H$-comodule homomorphism.  
In order to define a polynomial $H$-identity,  E.~Aljadeff and C.~Kassel use $\mathbbm k \langle X  | H \rangle$ too,
however considering on $\mathbbm k \langle X  | H \rangle$ a structure of an $H$-comodule algebra.
Namely, $H$ is now a right $H$-comodule where the comodule map $\rho$ coincides with the comultiplication $\Delta$ on $H$, the spaces $\mathbbm k \langle X \rangle^{(n)}$ are trivial $H$-comodules where $\rho(w):= w \otimes 1$
and structure of an $H$-comodule on $\mathbbm k \langle X  | H \rangle$ is induced from $H$ and $\mathbbm k \langle X \rangle^{(n)}$
via the tensor product $\otimes$.
An element $f\in \mathbbm k \langle X  | H \rangle$ is a polynomial $H$-identity in $A$ if $\varphi(f)=0$
for all homomorphisms $\varphi \colon \mathbbm k \langle X  | H \rangle \to A$ of algebras and $H$-comodules.

\begin{example} Let $G$ be a group and let $\mathbbm k G$ be its group Hopf algebra.
	Recall that the comultiplication $\Delta \colon \mathbbm k G \to \mathbbm k G \otimes \mathbbm k G$,
	the counit $\varepsilon \colon \mathbbm k G \to \mathbbm k$ and the antipode $S  \colon \mathbbm k G \to \mathbbm k G$
	 are defined on $\mathbbm k G$ as follows: $\Delta g := g \otimes g$, $\varepsilon(g)=1$, $Sg:= g^{-1}$ for all $g\in G$.
	Moreover, right $\mathbbm k G$-modules are just $G$-graded vector spaces $M = \bigoplus\limits_{g\in G} M^{(g)}$
	where $\rho \colon M \to M \otimes  \mathbbm k G$ is defined by $\rho(m)=m\otimes g$ for $m\in M$, $g\in G$.
	Let $x^{(g)}_i := g \otimes x_i \in \mathbbm k \langle X  | \mathbbm k G \rangle$ for $g\in G$ and $i\in\mathbb N$. Then $\mathbbm k \langle X  | \mathbbm k G \rangle$ can be identified with the free associative non-unital algebra $\mathbbm k \langle X^{G\text{-}\mathrm{gr}} \rangle$ where $X^{G\text{-}\mathrm{gr}} := \left\lbrace\left. x^{(g)}_i \right|  g\in G,\ i\in\mathbb N\right\rbrace$.
	Expanding the definition of polynomial $\mathbbm k G$-identity for a $G$-graded algebra $A=\bigoplus\limits_{g\in G} A^{(g)}$,
	we obtain that $f \in \mathbbm k \langle X^{G\text{-}\mathrm{gr}} \rangle$ is a polynomial $\mathbbm k G$-identity
	if and only if $\varphi(f)=0$ for all algebra homomorphisms $\varphi \colon \mathbbm k \langle X^{G\text{-}\mathrm{gr}} \rangle \to A$
	such that $\varphi(X^{(g)}) \subseteq A^{(g)}$, which is exactly the definition of a \textit{$G$-graded polynomial identity}.
	Denote by $\Id^{G\text{-}\mathrm{gr}}(A)$ the set of $G$-graded polynomial identities of $A$, which is a graded ideal of 
	$\mathbbm k \langle X^{G\text{-}\mathrm{gr}} \rangle$.
The subspace $$P_n^{G\text{-}\mathrm{gr}} = \left\langle\left. x_{\sigma(1)}^{(g_1)} \ldots  x_{\sigma(n)}^{(g_n)} \right| \sigma\in S_n,\ g_i \in G \right\rangle_{\mathbbm k} \subset \mathbbm k \langle X^{G\text{-}\mathrm{gr}} \rangle$$
is called the space of \textit{multilinear $G$-graded polynomials} in variables $x_1, \ldots, x_n$.
The dimension $c_n^{G\text{-}\mathrm{gr}}(A) := \dim\frac{P_n^{G\text{-}\mathrm{gr}} }{\strut P_n^{G\text{-}\mathrm{gr}}  \cap\, \Id^{G\text{-}\mathrm{gr}}(A)}$ is called the \textit{$n$th codimension
	of $G$-graded polynomial identities} of  $A$.
\end{example}

When $H$ is different from a group algebra, the description of all $H$-comodule algebra homomorphisms
$\mathbbm k \langle X  | H \rangle \to A$ may become not as explicit as in the case of $H$-module algebras.
The main difference is the following. While  $\mathbbm k \langle X  | H \rangle$ is the tensor
algebra of the free $H$-module $H\otimes \langle x_1, x_2, x_3\ldots\rangle_{\mathbbm k}$ and
all $H$-module algebra homomorphisms from  $\mathbbm k \langle X  | H \rangle$
are defined by images of $x_i$, the same vector space
$H\otimes \langle x_1, x_2, x_3\ldots\rangle_{\mathbbm k}$ is not free, but
a cofree right $H$-comodule. 

On the other hand, every right $H$-comodule is a left $H^*$-module where $H^*$ is the algebra dual to the coalgebra $H$.
Namely, $\gamma m := \gamma(m_{(1)})m_{(0)}$ for all $\gamma \in H^*$, $m\in M$ and an $H$-comodule $M$.
Moreover, all $H^*$-module homomorphisms between two comodules are $H$-comodule homomorphisms
and vice versa. The dual $H^*$ of a finite dimensional Hopf algebra $H$ is again a Hopf algebra. Moreover,
by the Larson~--- Sweedler theorem,
$H$ is co-Frobenius, that is $H \cong H^*$ as left $H^*$-modules (see e.g.~\cite{Lin} and~\cite[Theorem 2.1.3]{Montgomery}),
and $$\mathbbm k \langle X  | H \rangle \cong \mathbbm k \langle X  | H^* \rangle$$
as $H^*$-module algebras.

	Let $\mathcal C := \mathsf{Comod}^H$ for an arbitrary Hopf algebra $H$ over a field $\mathbbm k$,
	$\mathcal D := \mathbf{Vect}_{\mathbbm k}$ and let $U \colon \mathcal C \to \mathcal D$ be the corresponding
	forgetful functor.
	For an arbitrary natural transformation $\theta \colon U \Rightarrow U$ 
	define the linear map $\gamma \colon H \to \mathbbm k$ by
	$\gamma(h):= \varepsilon(\theta_H(h))$ for $h\in H$ where $H$ is considered to be a right module over itself
	via the comultiplication $\Delta \colon H \to H \otimes H$.
	Let $M$ be a right $H$-comodule and let $\alpha \colon M \to  \mathbbm k$ be an arbitrary linear function.
	Denote by $\rho \colon M \to M \otimes H$ the comodule structure map on $M$.
	Then $(\alpha \otimes \id_H) \rho \colon M \to H$ is a comodule homomorphism.
	Therefore the diagram below is commutative by the naturality of $\theta$:
	$$\xymatrix{
		M \ar[rr]^{\theta_M} \ar[d]_{(\alpha \otimes \id_H) \rho} & \quad\quad & M \ar[d]_{(\alpha \otimes \id_H) \rho}\\
		 H \ar[rr]^{\theta_H}  & & H  \\
		}$$ 
	
	Hence $(\alpha \otimes \id_H) \rho \theta_M = \theta_H (\alpha \otimes \id_H) \rho$.
	Applying the counit $\varepsilon \colon H \to \mathbbm k$ to both sides and substituting an arbitrary element $m$ for the argument, we get
	$$(\alpha \otimes \varepsilon) \rho \theta_M (m) = \gamma(m_{(1)})\alpha(m_{(0)}),$$
	$$\alpha(\theta_M (m)_{(0)}) \varepsilon(\theta_M (m)_{(1)}) = \alpha \bigl( \gamma(m_{(1)})m_{(0)} \bigr),$$
	$$\alpha(\theta_M (m)) = \alpha \bigl(\gamma(m_{(1)})m_{(0)} \bigr).$$

	Since $\alpha \in M^*$ is arbitrary, we get $\theta_M (m) = \gamma(m_{(1)})m_{(0)}$.
	Conversely, every $\gamma\in H^*$ defines the natural transformation $\theta \colon U \Rightarrow U$ 
	by $\theta_M (m) := \gamma(m_{(1)})m_{(0)}$ for $m\in M$ where $M$ is a right $H$-comodule.
		Therefore, $\End(U)$ can be identified with $H^*$.
	
	Let $\Omega=\lbrace \mu \rbrace$, $s(\mu)=2$, $t(\mu)=1$.
	If $H$ is finite dimensional, then, as we have mentioned above, $H \cong H^*$ as left $H^*$-modules
	and  $\mathcal P_V(\mathbbm{k}, \lbrace \mu \rbrace \sqcup H^*)(n, 1)$
	for $V = V_\mathrm{assoc} \cup \lbrace h(xy)-(h_{(1)} x)(h_{(2)} y) \mid h\in H^*\rbrace$
	can be identified with the subspace of $\mathbbm k \langle X  | H \rangle$ consisting of all multilinear $H$-polynomials
	in $x_1, \ldots, x_n$. Again, under this identification polynomial identities correspond to polynomial identities.
	 In particular, for an algebra $A$ graded by a finite group $G$
	we have the equality of codimensions
	$$\dim \mathcal E_{UA}\bigl(\mathcal P(\mathbbm{k}, \lbrace \mu \rbrace \sqcup (\mathbbm{k}G)^*)(n,1) \bigr) = c_n^{G\text{-}\mathrm{gr}}(A).$$
	 By~\cite[Lemma 7.1]{ASGordienko9},	the equality above still holds if $G$ is infinite but $A$ is finite dimensional.

\section{Examples and applications}\label{SectionExamplesApplications}

\subsection{Vector spaces}

Every object $A$ in a braided monoidal category $\mathcal C$ can be considered as an $\varnothing$-magma.
However, the category $\mathcal M(\varnothing)$ still has non-trivial morphisms resulting from the braiding in $\mathcal C$.
At the same time, $\mathcal M(\varnothing)(m,n)=\varnothing$ for $m\ne n$.

Let $W$ be a vector space over a field $\mathbbm k$.
Recall that the category $\mathbf{Vect}_{\mathbbm k}$ is symmetric, which implies that 
the polynomial identity $\tau_{1,1}^2-\id_2\equiv 0$
 holds in $W$.
 By the Coherence Theorem for the symmetric categories applied to
 $\mathcal M_{V_\mathrm{symm}}(\varnothing)$, where $V_\mathrm{symm} := \lbrace\tau_{1,1}^2-\id_2\rbrace$ (see Remark~\ref{RemarkTau11}), there exists an isomorphism
between the $n$th symmetric group $S_n$ and the monoid 
 $\mathcal M_{V_\mathrm{symm}}(\varnothing)(n,n)$ where the
morphism $f \colon n \to n$  corresponds to the permutation $\sigma \in S_n$ such that
$$(\mathcal E_Wf)(a_1 \otimes a_2 \otimes \dots \otimes a_n) = a_{\sigma^{-1}(1)} \otimes a_{\sigma^{-1}(2)} \otimes \dots \otimes a_{\sigma^{-1}(n)}
\text{ for all } a_1, \ldots, a_n \in W.$$
This bijection extends to the isomorphism $\theta \colon \mathbbm{k} S_n \mathrel{\widetilde\to} \mathcal P_{V_\mathrm{symm}}(\mathbbm k, \varnothing)(n,n)$
of algebras.

\begin{theorem} Let $W$ be a vector space over a field $\mathbbm k$, $\mathop\mathrm{char} \mathbbm{k} = 0$, $\dim W = d \in\mathbb Z_+$.
Then all polynomial
$\varnothing$-identities in $W$ with coefficients in~$\mathbbm k$
follow from the set
\begin{equation*}\begin{split}
		V_{\mathrm{symm},d} := V_\mathrm{symm} \cup 
		\left\lbrace
		\sum\limits_{\sigma\in S_{d+1}}
		(\sign\sigma)
		\theta(\sigma)\right\rbrace.\end{split}\end{equation*}
Moreover $c_{m,n}(W) = 0$ for $m\ne n$ and $$ c_{n,n}(W)
 \sim \alpha_d n^{-\frac{d^2-1}{2}}
d^{2n} \text{ as }n\to\infty$$
where $\alpha_d := \left(\frac{1}{\sqrt{2\pi}}\right)^{d-1}
\left(\frac{1}{2}\right)^{(d^2-1)/2} \cdot 1! \cdot 2! \cdot\ldots
\cdot (d-1)! \cdot d^{d^2/2}$.
\end{theorem}
\begin{proof}
One of the main tools in the investigation of polynomial
identities is provided by the representation theory of symmetric groups. Recall that irreducible representations of $S_n$ are described by partitions $\lambda =(\lambda_1, \ldots, \lambda_s) \vdash n$ where $\lambda_1 \geqslant \lambda_2 \geqslant \dots \geqslant
\lambda_s \geqslant 0$, $\sum\limits_{i=1}^s \lambda_i = n$,
and their diagrams $D_\lambda$. Let  $e_{T_{\lambda}}=a_{T_{\lambda}} b_{T_{\lambda}}$
 and
 $e^{*}_{T_{\lambda}}=b_{T_{\lambda}} a_{T_{\lambda}}$
 where
 $a_{T_{\lambda}} = \sum\limits_{\pi \in R_{T_\lambda}} \pi$
 and
 $b_{T_{\lambda}} = \sum\limits_{\sigma \in C_{T_\lambda}}
 (\sign \sigma) \sigma$,
 be the Young symmetrizers corresponding to a Young tableau~$T_\lambda$.
  Then $M(\lambda) = \mathbbm kS_n e_{T_\lambda} \cong \mathbbm k S_n e^{*}_{T_\lambda}$
 is an irreducible $\mathbbm kS_n$-module corresponding to
 a partition~$\lambda \vdash n$. We refer the reader to~\cite{Bahturin, DrenKurs, GiaZai}
 for an account of $S_n$-representations and their applications to polynomial
 identities.

 The polynomial identity
$\sum\limits_{\sigma\in S_{d+1}}
(\sign\sigma)
\theta(\sigma)\equiv 0$ 
holds in $W$ since $\dim W = d$.
Therefore, $\mathcal{E}_W\theta$ factors through $\mathcal P_{V_{\mathrm{symm},d}}(\mathbbm k, \varnothing)(n,n)$. Hence all Young symmetrizers corresponding to Young diagrams $D_\lambda$ of height $\mathop\mathrm{ht} \lambda$ greater than $d$
are mapped by $\mathcal{E}_W\theta$ to $0$. On the other hand,
given a Young tableau $T_\lambda$ where $\mathop\mathrm{ht} \lambda \leqslant d$, apply
 $\mathcal{E}_W\theta\left(e^{*}_{T_{\lambda}}\right)$ to the tensor product
 $a_1 \otimes \dots \otimes a_n$
  of basis elements $b_1, \ldots, b_d$ of $W$ where $a_i := b_j$ if the number $i$ appears in the $j$th row of $T_\lambda$.
  Then the result is nonzero, whence $\mathcal{E}_W\theta\left(e^{*}_{T_{\lambda}}\right)\ne 0$. 
  
  Recall that the group algebra $\mathbbm k S_n$ is the direct sum of minimal ideals,
  each of which is the direct sum of isomorphic irreducible representations of $S_n$.
  Every Young symmetrizer belongs to the corresponding minimal ideal.
  Thus $\Ker\left(\mathcal{E}_W\theta\right)$ is generated as an ideal by Young symmetrizers $e^{*}_{T_{\lambda}}$ with $\mathop\mathrm{ht} \lambda > d$
  and the image of $\Ker\left(\mathcal{E}_W\theta\right)$ in $\mathcal P_{V_{\mathrm{symm},d}}(\mathbbm k, \varnothing)(n,n)$ is zero. Therefore, all polynomial identities of $W$ are indeed generated by the set $V_{\mathrm{symm},d}$.
Moreover, the image of $\mathcal{E}_W\theta$ is isomorphic as a vector space to the direct sum of $M(\lambda)$ 
taken with the same multiplicity as in $\mathbbm k S_n$, i.e. $\dim M(\lambda)$, where $\lambda \vdash n$,
$\mathop\mathrm{ht} \lambda \leqslant d$.
Hence $$c_{n,n}(W)=\sum\limits_{\substack{\lambda \vdash n, \\ \mathop\mathrm{ht} \lambda \leqslant d}} \bigl(\dim M(\lambda)\bigr)^2.$$
By~\cite[Section 4.5, Case 2]{RegevStrip}, $c_{n,n}(W)$ has the required asymptotic behavior.
\end{proof}
\begin{remark}
	If $d:=\dim W =0$, then
	$\sum\limits_{\sigma\in S_{d+1}}
	(\sign\sigma)
	\theta(\sigma)\equiv 0$ reduces to $\id_1 \equiv 0$ and $c_{m,n}(W)=0$ for all $m,n\in\mathbb Z_+$.
\end{remark}	

\subsection{Hopf algebras}\label{SubsectionHopfAlgebras}
Every Hopf algebra $H$ over a field $\mathbbm k$ is an $\Omega$-magma in $\mathbf{Vect}_{\mathbbm k}$
where $\Omega=\lbrace \mu, u, \Delta, \varepsilon, S \rbrace$ are symbols for
the multiplication $\mu_H \colon H \otimes H \to H$, the unit $u_H \colon \mathbbm k \to H$,
the comultiplication $\Delta_H \colon H  \to H \otimes H$,
the counit $\varepsilon_H \colon H  \to \mathbbm k$ and the antipode $S_H \colon H \to H$.
A vector space $H$ endowed with the maps listed above
is a Hopf algebra if the diagrams corresponding to the conditions from Section~\ref{SubsectionMonComonHopfMon} are commutative. Since the commutativity of each diagram means that the corresponding polynomial identity holds in $H$,
Hopf algebras form a variety $\mathrm{Var}(V_\mathbf{Hopf})$ where 
\begin{equation}\label{EqVHopf}\begin{split}
V_\mathbf{Hopf}=V_\mathrm{symm} \cup \lbrace\mu(\mu \otimes \id_1)-\mu(\id_1 \otimes \mu ),\quad
\mu (u\otimes \id_1) - \id_1,\quad
\mu (\id_1 \otimes u) - \id_1,\\
(\Delta \otimes \id_1)\Delta - (\id_1 \otimes \Delta)\Delta,\quad
(\varepsilon \otimes \id_1)\Delta - \id_1, \quad
(\id_1 \otimes \varepsilon)\Delta - \id_1, \\
 \varepsilon u - \id_0, \quad
 \varepsilon\mu - \varepsilon \otimes \varepsilon, \\
 \Delta u - u \otimes u,\quad
 \Delta\mu - (\mu \otimes \mu)(\id_1 \otimes \tau_{1,1} \otimes \id_1)(\Delta \otimes \Delta),\\
 \mu(S\otimes \id_1)\Delta - u\varepsilon,\quad
 \mu(\id_1 \otimes S)\Delta - u\varepsilon\rbrace.
 \end{split}
 \end{equation}

 Let $H$ be a Hopf algebra over a field $\mathbbm k$.
 While maps belonging to 
 $\mathcal E_H\bigl(\mathcal M(\Omega)(m, n) \bigr)$ can be depicted using graphical calculus (see,  e.g.,~\cite{KasselQuantumGroups}), their compositions can also be represented in the usual functional notation.
 For example, $\mu(\mu \otimes \id_1)$ can be written as $(xy)z$ or $(x_1 x_2)x_3$ and 
 $(\Delta \otimes \id_1)\Delta$ can be written as $x_{(1)(1)}\otimes x_{(1)(2)}\otimes x_{(2)}$.

A Hopf monoid $H$ is \textit{cocommutative} if the polynomial identity $x_{(1)}\otimes x_{(2)}\equiv x_{(2)}\otimes x_{(1)}$
holds in $H$. Denote the category $\mathcal P_{V_\mathbf{CocHopf}}(\mathbbm{k}, \Omega)$, where $$V_\mathbf{CocHopf} := V_\mathbf{Hopf} \cup \lbrace 
x_{(1)}\otimes x_{(2)} - x_{(2)}\otimes x_{(1)}\rbrace,$$
by $\mathcal P_{\mathbf{CocHopf}}(\mathbbm{k})$.
Then $1$ is a cocommutative Hopf monoid in $\mathcal P_{\mathbf{CocHopf}}(\mathbbm{k})$.
The standard convolution techniques (see e.g.~\cite[Proposition~4.2.7]{DNR}) show that
$S^2=\id_1$. Using the compatibility conditions and properties of the antipode $S$, in every monomial $f$ in $\mathcal P_{\mathbf{CocHopf}}(\mathbbm{k})$ 
we move the multiplication, the antipode and the unit to the codomain and the comultiplication and the counit to the domain.
We get
 $$f(x_1, \ldots, x_m) = \varepsilon(x_{i_1}) \ldots \varepsilon(x_{i_s}) \bigotimes\limits_{k=1}^n   S^{\alpha_{k1}}x_{j_{k1}(\ell_{k1})} \ldots S^{\alpha_{kt_k}}x_{j_{kt_k}(\ell_{kt_k})}$$
where $\alpha_{kr} \in \lbrace 0,1 \rbrace$, $i_1 < \ldots < i_s$, every symbol $x_q$ or $x_{q(r)}$ for given $1\leqslant q \leqslant m$ and $r \in \mathbb Z_+$ appears 
 no more than once and  there exist numbers $p_q \in \mathbb Z_+$
such that \begin{equation*}\begin{split}
  \lbrace  x_{i_1}, \ldots, x_{i_s}  \rbrace
\sqcup		
		\bigsqcup\limits_{k=1}^n  \lbrace  x_{j_{k1}(\ell_{k1})}, \ldots, x_{j_{kt_k}(\ell_{kt_k})}  \rbrace
\\=\lbrace x_q \mid 1\leqslant q \leqslant m,\ p_q = 0 \rbrace
\sqcup \lbrace x_{q(r)} \mid  1\leqslant q \leqslant m,\  p_q > 0,\ 1\leqslant r\leqslant p_q  \rbrace.\end{split}\end{equation*}
Above we use the Sweedler notation
$$x_{q(1)}\otimes \dots \otimes x_{q(p_q)} := \left\lbrace\begin{array}{cc}(\Delta \otimes \id_{p_q-2})\ldots(\Delta \otimes \id_1)\Delta x_q &
\text{ if } p_q \geqslant 2, \\
x_q &\text{ if }p_q = 1. \end{array}\right.$$

By the cocommutativity, we may assume that if $j_{k\alpha} = j_{q\beta}$
but either $k < q$ or $k=q$ and $\alpha < \beta$, then $\ell_{k\alpha} < \ell_{q\beta}$. In addition, using the identity $h_{(1)}Sh_{(2)}=(Sh_{(1)})h_{(2)}=\varepsilon(h)$ for all $h\in H$,
we exclude the entries $\left(S x_{q(r)}\right) x_{q(r+1)}$ and
$x_{q(r)} S x_{q(r+1)}$.

Now we define the linear map
$\varphi_{m,n} \colon \mathbbm{k}\mathcal F (x_1, \dots, x_m)^{{}\otimes n} \to \mathcal P_{\mathbf{CocHopf}}(\mathbbm{k})(m,n)$ where $m,n \in\mathbb Z_+$ and $\mathcal F (X)$ is the free group on a set $X$.
Consider the standard basis  in $\mathbbm{k}\mathcal F (x_1, \ldots, x_m)^{{}\otimes n}$ consisting of the elements $w_1 \otimes w_2 \otimes \dots \otimes w_n$ where $w_i \in \mathcal F (x_1, \ldots, x_m)$, $1\leqslant i \leqslant n$.
By the definition, $\varphi_{m,n}(w_1 \otimes w_2 \otimes \dots \otimes w_n)$
is obtained from $w_1 \otimes w_2 \otimes \dots \otimes w_n$ as follows:
\begin{enumerate}
	\item we replace all $x_j^{-1}$ with $Sx_j$;
	\item if some $x_j$ appears more than once, we replace each $x_j$ with $x_{j(r)}$ where $r$ is the number of the entry;
	\item if some $x_j$ does not appear, we add $\varepsilon(x_j)$ in the beginning of the expression.
\end{enumerate}
 The map $\varphi_{m,n}$ is extended on $\mathbbm{k}\mathcal F (x_1, \dots, x_m)^{{}\otimes n}$ by the linearity.
\begin{example}
	$$\varphi_{5,2}\left(x_1 x_2 x_1^{-1} \otimes x_5 x_2^{-1} x_1^2\right) = \varepsilon(x_3)\varepsilon(x_4)\, x_{1(1)}
	x_{2(1)} S x_{1(2)} \otimes x_5 \left( Sx_{2(2)}\right) x_{1(3)} x_{1(4)}.$$
\end{example}

\begin{theorem}
	For every $m,n\in \mathbb Z_+$ the map $\varphi_{m,n}$ is a linear bijection. Moreover, if $\mathcal F ( x_1, x_2, \ldots )$ is the free group
	of the countable rank, then the (multilinear) polynomial identities  
	with coefficients in~$\mathbbm k$ of $\mathbbm{k}\mathcal F ( x_1, x_2, \ldots)$ as a Hopf algebra  are generated by the identity $x_{(1)}\otimes x_{(2)} - x_{(2)}\otimes x_{(1)}\equiv 0$ of cocommutativity.
\end{theorem}
\begin{proof}
	The surjectivity of $\varphi_{m,n}$ follows from the remarks made above.
	Replacing every variable $x_j$ in an element of $\mathcal P_{\mathbf{CocHopf}}(\mathbbm{k})(m,n)$
	with the generator $x_j$ of $\mathcal F ( x_1, x_2, \ldots )$,
	we see that the images of $w_1 \otimes w_2 \otimes \dots \otimes w_n$
	under the composition $\mathcal E_{\mathcal F ( x_1, x_2, \ldots )} \varphi_{m,n}$
	are linearly independent, which implies that $\varphi_{m,n}$ is injective.
	
	If there exists a polynomial identity of $\mathbbm{k}\mathcal F ( x_1, x_2, \ldots)$
	that does not follow from the cocommutativity,
	then  $\mathcal E_{\mathcal F ( x_1, x_2, \ldots )} f$ is
	the zero linear map
	for some $f \in \mathcal P_{\mathbf{CocHopf}}(\mathbbm{k})(m,n)$
	 and $m,n\in \mathbb Z_+$. The substitution made above implies $f=0$.
\end{proof}

\begin{remark}
	The linear bijections $\varphi_{m,n}$ provide a description of the category $\mathcal P_{\mathbf{CocHopf}}(\mathbbm{k})$.
\end{remark}

\subsection{Commutative cocommutative Hopf algebras}\label{SubsectionComCocomHopfAlgebras}

Again, let $\mathbbm{k}$ be a field.
Denote the category $\mathcal P_{V_\mathbf{CommCocHopf}}(\mathbbm{k}, \Omega)$, where $$V_\mathbf{CommCocHopf} := V_\mathbf{CocHopf} \cup \lbrace 
xy-yx\rbrace,$$
by $\mathcal P_{\mathbf{CommCocHopf}}(\mathbbm{k})$.
Then $1$ is a commutative cocommutative Hopf monoid in $\mathcal P_{\mathbf{CommCocHopf}}(\mathbbm{k})$.
Using the commutativity,
every monomial $f \in \mathcal P_{\mathbf{CommCocHopf}}(\mathbbm{k})$ 
can be presented in the following form:
\begin{equation}\label{EqCommCocHopfMonomial}f(x_1, \ldots, x_m) = \varepsilon(x_{i_1}) \ldots \varepsilon(x_{i_s}) \bigotimes\limits_{k=1}^n 
\prod\limits_{r=1}^m \prod\limits_{j=1}^{p_{kr}} S^{\alpha_{kr}} x_{r(\beta_{krj})}\end{equation}
where $p_{kr} \in \mathbb Z_+$, $\alpha_{kr} \in \lbrace 0,1 \rbrace$, $i_1 < \ldots < i_s$ and $$(\beta_{1r1}, \beta_{1r2}, \ldots, \beta_{2r1}, \ldots, \beta_{n,r, p_{nr}})
=\left(1, \ldots, \sum_{k=1}^n p_{kr} \right)$$ for every $1\leqslant r \leqslant m$,
$$\lbrace  x_{i_1}, \ldots, x_{i_s}  \rbrace = \left\lbrace r \left| 1\leqslant r \leqslant m,\  \sum_{k=1}^n p_{kr} = 0\right.\right\rbrace,$$
$x_{r(1)}:= x_r$ if $\sum\limits_{k=1}^n p_{kr} = 1$.

Now we define the linear map
$\psi_{m,n} \colon \mathbbm{k}\mathcal F_{\mathbf{Ab}} (x_1, \dots, x_m)^{{}\otimes n} \to \mathcal P_{\mathbf{CommCocHopf}}(\mathbbm{k})(m,n)$ where $m,n \in\mathbb Z_+$ and $\mathcal F_{\mathbf{Ab}} (X)$ is the free abelian group a set $X$ written in a multiplicative form.
The standard basis in $\mathbbm{k}\mathcal F_{\mathbf{Ab}} (x_1, \ldots, x_m)^{{}\otimes n}$ consists of elements
 $g=\bigotimes\limits_{k=1}^n 
 \prod\limits_{r=1}^m x_r^{s_{kr}}$
  where $s_{kr} \in \mathbb Z$.
By the definition, $\psi_{m,n}(g):= f$ where $f$ is defined in~\eqref{EqCommCocHopfMonomial} above  for
$p_{kr} := s_{kr}$ and $\alpha_{kr} := 0$ if  $s_{kr}\geqslant 0$ and $p_{kr} := -s_{kr}$ and $\alpha_{kr} := 1$ if  $s_{kr}< 0$.
The map $\psi_{m,n}$ is extended on $\mathbbm{k}\mathcal F_{\mathbf{Ab}} (x_1, \dots, x_m)^{{}\otimes n}$ by the linearity.
\begin{example} 
	$$\psi_{5,2}\left(x_1^{-2} x_2^{3} \otimes  x_1^2 x_2^{-1}  x_5 \right) = \varepsilon(x_3)\varepsilon(x_4)\,
	 (Sx_{1(1)}) \left(Sx_{1(2)}\right) x_{2(1)} x_{2(2)} x_{2(3)} \otimes
	 x_{1(3)} x_{1(4)} \left( Sx_{2(4)}\right) x_5.$$
\end{example}

\begin{theorem}
	For every $m,n\in \mathbb Z_+$ the map $\psi_{m,n}$ is a linear bijection. Moreover, if $\mathcal F_{\mathbf{Ab}} ( x_1, x_2, \ldots )$ is the free abelian group
	of the countable rank, then the (multilinear) polynomial identities with coefficients in~$\mathbbm k$ of $\mathbbm{k}\mathcal F_{\mathbf{Ab}} ( x_1, x_2, \ldots)$ as a Hopf algebra  are generated by the identities
	$xy-yx\equiv 0$ of commutativity and
	 $x_{(1)}\otimes x_{(2)} - x_{(2)}\otimes x_{(1)}\equiv 0$ of cocommutativity.
\end{theorem}
\begin{proof}
	The surjectivity of $\psi_{m,n}$ follows from the remarks made above.
	Replacing every variable $x_j$ in an element of $\mathcal P_{\mathbf{CommCocHopf}}(\mathbbm{k})(m,n)$
	with the generator $x_j$ of $\mathcal F_{\mathbf{Ab}} ( x_1, x_2, \ldots )$,
	we see that the images of basis elements $g$
	under the composition $\mathcal E_{\mathcal F_{\mathbf{Ab}} ( x_1, x_2, \ldots )} \psi_{m,n}$
	are linearly independent, which implies that $\psi_{m,n}$ is injective.
	
	If there exists a polynomial identity of $\mathbbm{k}\mathcal F ( x_1, x_2, \ldots)$
	that does not follow from the commutativity and cocommutativity,
	then  $\mathcal E_{\mathcal F_{\mathbf{Ab}} ( x_1, x_2, \ldots )} f$ is
	the zero linear map
	for some $f \in \mathcal P_{\mathbf{CommCocHopf}}(\mathbbm{k})(m,n)$
	and $m,n\in \mathbb Z_+$. The substitution made above implies $f=0$.
\end{proof}	

\begin{remark}
	The linear bijections $\psi_{m,n}$ provide a description of the category $\mathcal P_{\mathbf{CommCocHopf}}(\mathbbm{k})$.
\end{remark}

\subsection{Group Hopf algebra of the cyclic group of order $2$}
Let $C_2 = \lbrace 1, c \rbrace$, $c^2 =1$, be the cyclic group of order $2$ and let $\mathbbm{k}$ be a field, $\mathop\mathrm{char}\mathbbm{k} \ne 2$.
Then the following polynomial identities hold in the group algebra $\mathbbm kC_2$ of $C_2$:
\begin{equation}\label{EqCyclicOrderTwo}
\begin{split}
x_{(1)} \otimes x_{(2)} \equiv x_{(2)} \otimes x_{(1)},\\
xy \equiv yx,\\
x_{(1)}x_{(2)}\equiv \varepsilon(x)1,\\
\bigl( x_{(1)}-\varepsilon(x_{(1)}) 1 \bigr) \otimes \bigl( x_{(2)}-\varepsilon(x_{(2)})1 \bigr)
\otimes \bigl( y-\varepsilon(y)1 \bigr)\\ \equiv \bigl( x-\varepsilon(x)1 \bigr) \otimes \bigl( y_{(1)}-\varepsilon(y_{(1)})1 \bigr) 
\otimes \bigl( y_{(2)}-\varepsilon(y_{(2)})1 \bigr).
\end{split}
\end{equation}
(It is sufficient to substitute basis elements.)
Below we show that all polynomial identities in $\mathbbm kC_2$ follow from~\eqref{EqCyclicOrderTwo}.

Note that~\eqref{EqCyclicOrderTwo} implies
 \begin{equation}\label{EqCyclicOrderTwoSquare}\begin{split}
 \bigl( x_{(1)}-\varepsilon(x_{(1)}) 1 \bigr)\bigl( x_{(2)}-\varepsilon(x_{(2)})1 \bigr)
 \equiv -2 \bigl( x-\varepsilon(x)1 \bigr), \qquad
 Sx\equiv x.
 \end{split}
\end{equation}
Let $V_1$ be the union of $V_\mathbf{Hopf}$, defined in~\eqref{EqVHopf} above,
and the polynomials corresponding to~\eqref{EqCyclicOrderTwo}.
Consider the presentation~\eqref{EqCommCocHopfMonomial}.
Using $x_r = \varepsilon(x_r)1 + (x_r-\varepsilon(x_r)1)$, \eqref{EqCyclicOrderTwo} and \eqref{EqCyclicOrderTwoSquare},
every polynomial from $\mathcal P_{V_1}(\mathbbm{k}, \Omega)(m,n)$, where $m,n\in\mathbb Z_+$,
can be written as a linear combination of polynomials
\begin{equation}\label{EqCyclicOrderTwoMonomial}\begin{split}f(x_1, \ldots, x_m) = \varepsilon(x_{i_1}) \ldots \varepsilon(x_{i_s})
\,1 \otimes \dots \otimes 1 \otimes(x_{j_1(1)}-\varepsilon(x_{j_1(1)})1) \left(\prod\limits_{k=2}^t (x_{j_k}-\varepsilon(x_{j_k})1)\right) \\
\otimes 1 \otimes \ldots \otimes 1 \otimes (x_{j_1(2)}-\varepsilon(x_{j_1(2)})1) \\
\dots \\
\otimes 1 \otimes \ldots \otimes 1 \otimes (x_{j_1(q)}-\varepsilon(x_{j_1(q)})1) \\ \otimes 1 \otimes \dots \otimes 1
\end{split}
\end{equation} for $t\geqslant 1$, $q\geqslant 1$
and \begin{equation}\label{EqCyclicOrderTwoMonomialTriv}f(x_1, \ldots, x_m) = \varepsilon(x_1) \ldots \varepsilon(x_n)
\,1 \otimes \dots \otimes 1\end{equation} for $t=0$ where
$i_1 < \ldots < i_s$, $j_1 < \ldots < j_t$, $$\lbrace i_1, \ldots, i_s \rbrace \sqcup 
\lbrace j_1, \ldots, j_t \rbrace = \lbrace 1, \ldots, n \rbrace.$$

\begin{theorem} Let $\mathbbm{k}$ be a field, $\mathop\mathrm{char}\mathbbm{k} \ne 2$.
	Monomials~\eqref{EqCyclicOrderTwoMonomial} and~\eqref{EqCyclicOrderTwoMonomialTriv}
	are linearly independent modulo polynomial identities in $\mathbbm kC_2$ (as a Hopf algebra).
	As a consequence, (multilinear) polynomial identities with coefficients in~$\mathbbm k$ of $\mathbbm kC_2$ 
	are generated by~\eqref{EqCyclicOrderTwo}.
	Moreover, $c_{m,n}(\mathbbm kC_2) = 2^{m+n} - 2^m - 2^n +2$ for all $m,n\in\mathbb Z_+$. 
\end{theorem}	
\begin{proof} Suppose that for some $m,n\in\mathbb Z_+$ a linear combination
	of monomials~\eqref{EqCyclicOrderTwoMonomial} and~\eqref{EqCyclicOrderTwoMonomialTriv}
	is a polynomial identity. Among all the monomials that occur in this linear combination
	with non-zero coefficients, choose the monomial $f_0$ with the least $t=: t_0$
	and substitute $x_{i_k}=1$ for $1\leqslant k \leqslant s$, $x_{j_\ell}=c$ for $1\leqslant \ell \leqslant t_0$. Then the monomials with $t > t_0$ or different $j_\ell$ will be zero. The values of monomials with $1$ on different places in the tensor product~\eqref{EqCyclicOrderTwoMonomial}  will be linearly independent with the value of $f_0$.
	Thus $f_0$ must occur in the polynomial identity with the zero coefficient, and we get a contradiction.
	Hence monomials~\eqref{EqCyclicOrderTwoMonomial} and~\eqref{EqCyclicOrderTwoMonomialTriv}
	are linearly independent modulo polynomial identities in $\mathbbm kC_2$. Therefore, polynomial identities of $\mathbbm kC_2$
	are generated by~\eqref{EqCyclicOrderTwo}.
	
	The codimension $c_{m,n}(\mathbbm kC_2)$ equals the number of monomials~\eqref{EqCyclicOrderTwoMonomial} and~\eqref{EqCyclicOrderTwoMonomialTriv}. Thus
	$$c_{m,n}(\mathbbm kC_2) = 1 + \sum\limits_{t=1}^m \binom{m}{t} \sum\limits_{k=1}^n \binom{n}{k}
	= 1 + (2^n-1)(2^m-1)= 2^{m+n} - 2^m - 2^n +2.$$
\end{proof}	

\subsection{A Yetter~--- Drinfel'd module of dimension $2$ as an $\lbrace \sigma \rbrace$-magma in $\mathbf{Vect}_\mathbbm{k}$}
\label{SubsectionYDdim2sigmaVect}

Let $H$ be a Hopf algebra over a field $\mathbbm k$ with an invertible antipode $S$. Denote by ${}_H^H\mathcal{YD}$
the category of \textit{left Yetter~--- Drinfel'd modules} (or \textit{${}_H^H\mathcal{YD}$-modules} for short), i.e. left $H$-modules and $H$-comodules $M$ such that
the $H$-action and the $H$-coaction $\delta \colon M \to H \otimes M$ 
satisfy the following compatibility condition:
$$\delta(hm)=h_{(1)}m_{(-1)}Sh_{(3)} \otimes h_{(2)}m_{(0)} \text{ for every } m\in M \text{ and } h\in H.$$

The category ${}_H^H\mathcal{YD}$ is braided monoidal where the monoidal product of  ${}_H^H\mathcal{YD}$-modules
$M$ and $N$ is their usual tensor product $M\otimes N$ over $\mathbbm k$ with the induced structures of a left $H$-module and a left $H$-comodule. The braiding $\sigma_{M,N} \colon M\otimes N \to N \otimes M$ is defined
by the formula $$\sigma_{M,N}(m\otimes n) := m_{(-1)}n \otimes m_{(0)}\text{ for }m\in M,\ n\in N.$$

Every Yetter~--- Drinfel'd module (or, more generally, a braided vector space) $M$  can be considered an $\lbrace \sigma \rbrace$-algebra over $\mathbbm k$
where $\sigma_M \colon M \otimes M \to M \otimes M$ is just the braiding $\sigma_{M,M}$.

Let $\mathop\mathrm{char} \mathbbm k \ne 2$ and $M := \langle a,b \rangle_{\mathbbm k}$.
Define on $M$ the structure of a Yetter~--- Drinfel'd module over the Hopf algebra $\mathbbm kC_2$
by
$\delta a := 1\otimes a $, $\delta b := c\otimes b$,
$ca:= -a$, $cb=b$.

Below we calculate $c_{m,n}(M)$ where $M$ is considered a $\lbrace \sigma \rbrace$-algebra over $\mathbbm k$.
In addition to the Yetter~--- Drinfel'd braiding $\sigma_M$, which is now an operation,
in $\mathcal P(\mathbbm k, \lbrace \sigma\rbrace)$ there are 
ordinary swaps $\tau_{m,n} \colon m+n \to n+m$ where  
$\mathcal{E}_M \left(\tau_{m,n}\right)(u\otimes v) := v \otimes u$ for all $u\in M^{{}\otimes m}$
and $v\in M^{{}\otimes n}$, $m,n\in\mathbb Z_+$.

Again, define the embedding $\theta \colon \mathbbm{k} S_n \hookrightarrow \mathcal P_{V_\mathrm{symm}}(\mathbbm{k}, \lbrace \sigma \rbrace)(n,n)$ in the way that 
$$\left(\mathcal E_M\theta(\rho)\right)(a_1 \otimes a_2 \otimes \dots \otimes a_n) = a_{\rho^{-1}(1)} \otimes a_{\rho^{-1}(2)} \otimes \dots \otimes a_{\rho^{-1}(n)}
\text{ for all } a_1, \ldots, a_n \in M\text{ and } \rho\in S_n.$$

Let $p^{(n)}_{21}:=\frac{1}{2}(\id_2 - \tau_{1,1}\sigma)\otimes \id_{n-2}$ and  $p^{(n)}_{ij}:=\theta(\rho)p^{(n)}_{21} \theta(\rho)^{-1}$ where $\rho\in S_n$ such that $\rho(2)=i$, $\rho(1)=j$,
$ 1 \leqslant i,j \leqslant n$,
$i\ne j$, $n\in\mathbb N$, $n\geqslant 2$. Then \begin{equation}\label{EqYDdim2SigmaVectPij}\left(\mathcal E_M p^{(n)}_{ij}\right)(u_1 \otimes u_2 \otimes \dots \otimes u_n)
=\left\lbrace\begin{array}{cl}
u_1 \otimes u_2 \otimes \dots \otimes u_n & \text{if }
u_i = a\text{ and } u_j = b, \\
0 & \text{otherwise}
\end{array}
\right.\end{equation}
for every $u_1,\dots,u_n\in \lbrace a,b \rbrace.$

For every $\varnothing \subsetneqq I \subsetneqq \lbrace 1,2,\dots,n\rbrace$
define $p^{(n)}_I := \prod\limits_{\substack{i \in I,\\ j\in \bar I}}
p^{(n)}_{ij}$ where $\bar I := \lbrace j \mid 1\leqslant j \leqslant n,\ j\notin I \rbrace$.
Let $$p^{(n)}_\varnothing :=
\id_n - \sum\limits_{\varnothing \subsetneqq I \subsetneqq \lbrace 1,2,\dots,n\rbrace} p^{(n)}_I.$$

\begin{theorem} Let  $\mathop\mathrm{char} \mathbbm k \ne 2$ and let $M = \langle a,b \rangle_{\mathbbm k}$
	be the Yetter~--- Drinfel'd module over $\mathbbm kC_2$ defined above. Then polynomial identities with coefficients in~$\mathbbm k$ of $M$
as an $\lbrace \sigma \rbrace$-magma in $\mathbf{Vect}_\mathbbm{k}$
	  are generated by the set
	\begin{equation*}\begin{split}
V_2 := V_\mathrm{symm} \cup \left\lbrace
p^{(n)}_{i_1 j_1}p^{(n)}_{i_2 j_2}- p^{(n)}_{i_2 j_2}p^{(n)}_{i_1 j_1}, \quad
p^{(n)}_{i_1 j_1}p^{(n)}_{i_2 j_2}- p^{(n)}_{i_1 j_2}p^{(n)}_{i_2 j_1},\right.
\\
p^{(n)}_{i_1 j_1} p^{(n)}_{i_2 j_2} p^{(n)}_{i_1 j_2}-
p^{(n)}_{i_1 j_1} p^{(n)}_{i_2 j_2},
\quad
p^{(n)}_{i_1 j_1}p^{(n)}_{i_2 i_1},\quad
\left(p^{(n)}_{i_1 j_1}\right)^2-p^{(n)}_{i_1 j_1},\\
p^{(n)}_{i_1 j_1}  p^{(n)}_{i_2 j_2}\theta\bigl((i_1 i_2)\bigr)-
p^{(n)}_{i_1 j_1} p^{(n)}_{i_2 j_2},\quad
p^{(n)}_{i_1 j_1}  p^{(n)}_{i_2 j_2}\theta\bigl((j_1 j_2)\bigr)-
p^{(n)}_{i_1 j_1} p^{(n)}_{i_2 j_2}
\\
\mid 1\leqslant i_1,j_1,i_2,j_2 \leqslant n,\ n=2,3,4;\text{ in all }p^{(n)}_{ij} \text{ we have } i\ne j
\rbrace \\
{} \cup \left\lbrace\left. p^{(n)}_\varnothing \theta(\rho) - p^{(n)}_\varnothing  \right| n\geqslant 2,\ \rho\in S_n \right\rbrace.
\end{split}\end{equation*}
Moreover  $c_{m,n}(M) = 0$ for $m\ne n$, $c_{0,0}(M)=c_{1,1}(M)=1$ and for $n\geqslant 2$
we have $$c_{n,n}(M) =  \binom{2n}{n}-1 \sim \frac{ 4^n}{\sqrt{\pi n}}
\text{ as }n\to\infty.$$
\end{theorem}	
\begin{proof}  
By~\eqref{EqYDdim2SigmaVectPij}, for every $\varnothing \subsetneqq I \subsetneqq \lbrace 1,2,\dots,n\rbrace$
\begin{equation}\label{EqYDdim2SigmaVectPI}\left(\mathcal E_M p^{(n)}_I\right)(u_1 \otimes u_2 \otimes \dots \otimes u_n)
=\left\lbrace\begin{array}{cc}
u_1 \otimes u_2 \otimes \dots \otimes u_n & \text{if }
u_i = a\text{ for all } i\in I\\
& \text{and }u_j = b\text{ for all } j\in \bar I, \\
0 & \text{otherwise}
\end{array}
\right.\end{equation}
for every $u_1,\dots,u_n\in \lbrace a,b \rbrace.$
Hence
\begin{equation}\label{EqYDdim2SigmaVectPnothing}\left(\mathcal E_M p^{(n)}_\varnothing\right)(u_1 \otimes u_2 \otimes \dots \otimes u_n)
=\left\lbrace\begin{array}{cc}
u_1 \otimes u_2 \otimes \dots \otimes u_n & \text{if }
u_i = a\text{ for all } 1\leqslant i \leqslant n\\
& \text{or }u_i = b\text{ for all } 
1\leqslant i \leqslant n, \\
0 & \text{otherwise}
\end{array}
\right.\end{equation}
for every $u_1,\dots,u_n\in \lbrace a,b \rbrace.$
By~\eqref{EqYDdim2SigmaVectPij} and~\eqref{EqYDdim2SigmaVectPnothing},
 the elements of the set $V_2$ are indeed polynomial identities in $M$.
 
Tensoring $p^{(m)}_{ij}$ with $\id_{n-m}$
and conjugating them by $\theta(\rho)$ where $\rho\in S_n$, we obtain that the polynomial identities appearing in the set $V_2$
only for $n\leqslant 4$ imply the same polynomial identities for $n > 4$.

Denote the images of morphisms from $\mathcal P_{V_\mathrm{symm}}(\mathbbm{k}, \lbrace \sigma \rbrace)(n,n)$ in $\mathcal P_{V_2}(\mathbbm{k}, \lbrace \sigma \rbrace) (n,n)$ by the same symbols.

The identities from $V_2$ imply that $p^{(n)}_I$ for $\varnothing \subsetneqq I \subsetneqq \lbrace 1,2,\dots,n\rbrace$ are orthogonal idempotents. Moreover,
$$p^{(n)}_I p^{(n)}_{ij}= p^{(n)}_{ij} p^{(n)}_I=\left\lbrace\begin{array}{ll}
p^{(n)}_I &\text{if }i\in I\text{ and }j\in \bar I, \\
0 &\text{otherwise.}
\end{array}
\right.
$$
In addition, $\theta(\rho) p^{(n)}_\varnothing \theta(\rho)^{-1} = p^{(n)}_\varnothing$ for all $\rho \in S_n$.

 Therefore, $p^{(n)}_\varnothing$ is a central idempotent in the algebra $\mathcal P_{V_2}(\mathbbm{k}, \lbrace \sigma \rbrace) (n,n)$ and is orthogonal to the other $p^{(n)}_I$.
 
 Let $I,J \subsetneqq \lbrace 1,2,\dots,n\rbrace$.
 Recall that every permutation $\rho\in S_n$ is a composition of transpositions. 
By the identities from $V_2$, we have
$p^{(n)}_I \theta(\rho)=p^{(n)}_I $ if
$\rho(I)=I$ and $$p^{(n)}_I \theta(\rho) p^{(n)}_J =
p^{(n)}_I \theta(\rho) p^{(n)}_J \theta(\rho)^{-1} \theta(\rho)  =p^{(n)}_I  p^{(n)}_{\rho(J)} \theta(\rho)^{-1}=0$$
if $\rho(J)\ne I$.
Hence
 the vector space $\left\langle\left. p^{(n)}_I \theta(\rho) p^{(n)}_J\right| \rho\in S_n\right\rangle_{\mathbbm k}$ is one-dimensional
if $|I|=|J|$ and zero otherwise. 

For every $I,J \subsetneqq \lbrace 1,2,\dots,n\rbrace$,
$|I|=|J|$,
fix some permutation $\rho_{I,J}\in S_n$ such that
$\rho_{I,J}(J)=I$. Rewriting $\sigma=\theta\bigl((12)\bigr)\left(\id_2-2p^{(2)}_{21}\right)$
and inserting $\id_n = \sum\limits_{I \subsetneqq \lbrace 1,2,\dots,n\rbrace} p^{(n)}_I$, we obtain that
$\mathcal P_{V_2}(\mathbbm{k}, \lbrace \sigma \rbrace) (n,n)$ is
the linear span of the $V_2$-relative $\lbrace \sigma\rbrace$-polynomials
\begin{equation}\label{EqBasisYDkC2}
p^{(n)}_I \theta(\rho_{I,J})p^{(n)}_J.\end{equation}

By~\eqref{EqYDdim2SigmaVectPI} and~\eqref{EqYDdim2SigmaVectPnothing}, the $V_2$-relative $\lbrace \sigma\rbrace$-polynomials~\eqref{EqBasisYDkC2} form a basis in $\mathcal P_{V_2}(\mathbbm{k}, \lbrace \sigma \rbrace) (n,n)$, the restriction of $\mathcal E_M$ on $\mathcal P_{V_2}(\mathbbm{k}, \lbrace \sigma \rbrace) (n,n)$ is an injection,
$V_2$
generates  $\lbrace \sigma\rbrace$-polynomial identities in $M$
with coefficients in~$\mathbbm k$
 and
$$c_{n,n}(M) = \sum\limits_{k=0}^{n-1}  \binom{n}{k}^2= \sum\limits_{k=0}^n  \binom{n}{k}^2 -1 = \binom{2n}{n}-1 \sim \frac{ 4^n}{\sqrt{\pi n}}
\text{ as }n\to\infty.$$
\end{proof}

\subsection{A Yetter~--- Drinfel'd module of dimension $2$ as an $\varnothing$-magma in ${}_{\mathbbm{k}C_2}^{\mathbbm{k}C_2}
\mathcal{YD}$}\label{SubsectionYDdim2VarnothingYD}

In this section we consider the Yetter~--- Drinfel'd module $M$ defined in Section~\ref{SubsectionYDdim2sigmaVect} above, but without swaps from $\mathbf{Vect}_\mathbbm{k}$.
The functor $\mathcal E_M$
now maps $\tau_{k,\ell}$ to the corresponding components of the
${}_{\mathbbm{k}C_2}^{\mathbbm{k}C_2}
\mathcal{YD}$-braiding.

Before treating our particular case, we first describe the category $M(\varnothing)$.

Let $$B_n := \langle \sigma_1, \ldots, \sigma_{n-1} \mid \sigma_i\sigma_j = \sigma_j\sigma_i\text{ for }
|i-j|\geqslant 2\text{ and }  \sigma_i\sigma_{i+1}\sigma_i = \sigma_{i+1}\sigma_i\sigma_{i+1}
\text{ for }1\leqslant i \leqslant n-2 \rangle$$
be the $n$th braid group, $B_0 = B_1 := \lbrace 1 \rbrace$.
The invertibility of the braiding implies that each $M(\varnothing) (n,n)$ is a group too. Below we show that
$M(\varnothing) (n,n) \cong B_n$.

Denote by $\mathcal B$ the braid category, i.e. the braided monoidal category where the set of objects coincides with $\mathbb Z_+$, $\mathcal B(m,n)=\varnothing$ for $m\ne m$ and $\mathcal B(n,n)= B_n$ for every $n\in \mathbb Z_+$.

Define the group homomorphisms $\theta \colon B_n \to \mathcal M(\varnothing) (n,n)$
 by $\theta(\sigma_i):=\id_{i-1}\otimes \tau_{1,1} \otimes \id_{n-i-1}$
for $1\leqslant i \leqslant n-1$, $n\in\mathbb Z_+$.

\begin{proposition}
	The maps $\theta \colon B_n \to M(\varnothing) (n,n)$ defined above are group isomorphisms for every $n\in\mathbb Z_+$.
	Their extensions $\mathbbm{k} B_n \to \mathcal P(\mathbbm k, \varnothing) (n,n)$ by the linearity, which we denote below by the same letter $\theta$, are algebra isomorphisms. The category $M(\varnothing)$ is isomorphic as a braided monoidal category to the braid category $\mathcal B$.
\end{proposition}
\begin{proof}
	By the properties of the braiding, all $\tau_{k,\ell}$ are compositions of $\id_s\otimes \tau_{1,1} \otimes \id_t$,
	$s+t+2 = k+\ell$, $s,t\in\mathbb Z_+$, whence $\theta$ is surjective.
	Note that $1$ is an $\varnothing$-magma in $\mathcal B$.
	Consider the braided strong monoidal functor $\mathcal E_1 \colon \mathcal M(\Omega) \to \mathcal B$
	from Proposition~\ref{PropositionOmegaMagmaBSMFunctorsEquivalence}.
	We have $\mathcal E_1 = \theta^{-1}$ on each $\mathcal M(\varnothing) (n,n)$.
	Therefore, the homomorphisms $\theta$ are in fact isomorphisms and the categories $\mathcal M(\varnothing)$ and $\mathcal B$ are isomorphic.	Now extend $\theta$ to algebra homomorphisms $\mathbbm{k} B_n \to \mathcal P(\mathbbm k, \varnothing) (n,n)$
	by the linearity. Then $\theta$ are algebra isomorphisms.
\end{proof}	

Denote by $PB_n$ the kernel of the surjective homomorphism $\bar{(\ )} \colon B_n \to S_n$ where $\sigma_i \mapsto (i,i+1)$.
Elements of $PB_n$ are called \textit{pure braids}. The subgroup $PB_n$ is the normal closure
of the elements $\sigma_1^2, \ldots, \sigma_{n-1}^2$.
Note that
\begin{equation}\label{EqYDdim2VarnothingYDThetaRho}
\left(\mathcal E_M\theta(\rho)\right)(u_1 \otimes u_2 \otimes \dots \otimes u_n) = \pm\, u_{\bar\rho^{-1}(1)} \otimes u_{\bar\rho^{-1}(2)} \otimes \dots \otimes u_{\bar\rho^{-1}(n)}
\end{equation}
 for all $u_1, \ldots, u_n \in \lbrace a, b\rbrace$ and $\rho\in B_n$.

For every $1\leqslant i < j \leqslant n$, where $n\geqslant 2$, choose some braid $\rho_{ij} \in B_n$
such that $\bar\rho_{ij} (1) = i$ and $\bar\rho_{ij} (2) = j$
and define 
$$q_{ij}^{(n)} := \frac{1}{2}\theta\left(\rho_{ij}\right)(\id_n-\tau_{1,1}^2\otimes \id_{n-2})\theta\left(\rho_{ij}\right)^{-1},\quad
r_{ij}^{(n)} := \frac{1}{2}\theta\left(\rho_{ij}\right)(\id_n+\tau_{1,1}^2\otimes \id_{n-2})\theta\left(\rho_{ij}\right)^{-1}.
$$
For our convenience, let $q_{ji}^{(n)} := q_{ij}^{(n)}$, $r_{ji}^{(n)} := r_{ij}^{(n)}$.

By~\eqref{EqYDdim2VarnothingYDThetaRho}, \begin{equation}\label{EqYDdim2nothingYDQij}\left(\mathcal E_M q^{(n)}_{ij}\right)(u_1 \otimes u_2 \otimes \dots \otimes u_n)
=\left\lbrace\begin{array}{cl}
u_1 \otimes u_2 \otimes \dots \otimes u_n & \text{if }
u_i \ne u_j, \\
0 & \text{otherwise}
\end{array}
\right.\end{equation}
and 
\begin{equation}\label{EqYDdim2nothingYDRij}\left(\mathcal E_M r^{(n)}_{ij}\right)(u_1 \otimes u_2 \otimes \dots \otimes u_n)
=\left\lbrace\begin{array}{cl}
u_1 \otimes u_2 \otimes \dots \otimes u_n & \text{if }
u_i = u_j, \\
0 & \text{otherwise}
\end{array}
\right.\end{equation}
for every $u_1,\dots,u_n\in \lbrace a,b \rbrace.$

Denote by $\Xi_n$ the set of all decompositions $\lbrace I, \bar I \rbrace$ where $I\sqcup \bar I = \lbrace 1,2,\dots,n\rbrace$, $I,\bar I \ne \varnothing$. Here we stress that $\lbrace I, \bar I \rbrace = \lbrace \bar I, I \rbrace$.

For every $\lbrace I, \bar I \rbrace \in \Xi_n$
define $$q^{(n)}_{\lbrace I, \bar I \rbrace} := \left( \prod\limits_{\substack{i \in I,\\ j \in \bar I}} q^{(n)}_{i j} \right) \left( \prod\limits_{\substack{i,j \in I,\\ i < j}} 
r^{(n)}_{ij} \right) \left( \prod\limits_{\substack{i,j \in \bar I,\\ i < j}} 
r^{(n)}_{ij} \right).$$
Let $$q^{(n)}_{\lbrace \varnothing, \bar\varnothing \rbrace} :=
\id_n - \sum\limits_{\lbrace I, \bar I \rbrace \in \Xi_n} q^{(n)}_{\lbrace I, \bar I \rbrace}$$
where for brevity we use the notation $\bar\varnothing := \lbrace 1,2,\dots,n\rbrace$.

Let $\tilde \Xi_n := \Xi_n \cup \bigl\lbrace\lbrace \varnothing, \bar\varnothing \rbrace \bigr\rbrace$.
Then, by~\eqref{EqYDdim2nothingYDQij} and~\eqref{EqYDdim2nothingYDRij},
\begin{equation}\label{EqYDdim2nothingYDQIbarI}\left(\mathcal E_M q^{(n)}_{\lbrace I, \bar I \rbrace}\right)(u_1 \otimes u_2 \otimes \dots \otimes u_n)
=\left\lbrace\begin{array}{cl}
 & \text{if }
u_i = u_j \text{ for all }i,j\in I, \\
u_1 \otimes u_2 \otimes \dots \otimes u_n & u_i = u_j \text{ for all }i,j\in \bar I, \\
& u_i \ne u_j \text{ for all }i\in I,\ j\in \bar I, \\
\\
0 & \text{otherwise}
\end{array}
\right.\end{equation}
for every $\lbrace I, \bar I \rbrace \in \tilde \Xi_n$.

\begin{theorem} Let  $\mathop\mathrm{char} \mathbbm k \ne 2$ and let $M = \langle a,b \rangle_{\mathbbm k}$
	be the Yetter~--- Drinfel'd module over $\mathbbm kC_2$ defined in Section~\ref{SubsectionYDdim2sigmaVect}. Then polynomial identities with coefficients in~$\mathbbm k$ of $M$
	as an $\varnothing$-magma in ${}_{\mathbbm{k}C_2}^{\mathbbm{k}C_2}\mathcal{YD}$
	  are generated by the set
	\begin{equation*}\begin{split}
			V_3 := \left\lbrace \tau_{1,1}^4-\id_2 \right\rbrace \cup			
			\left\lbrace \theta\left(\sigma_i^2 \rho - \rho \sigma_i^2 \right) \mid 
			i=1,\ldots, n-1,\ \rho\in PB_n,\ n\geqslant 2
			 \right\rbrace \\
			  \cup \left\lbrace\left. q^{(n)}_{ij} \theta(\rho) -  \theta(\rho) q^{(n)}_{ij}
			  \right| 1\leqslant i<j \leqslant n,\ \rho\in B_n,\ n\geqslant 2,\ \bar\rho(i)=i,\  \bar\rho(j)=j\right\rbrace \\
			  \cup \left\lbrace q^{(n)}_{\lbrace\lbrace 1,\dots,i \rbrace, \lbrace i+1,\dots, n \rbrace\rbrace}\left(\id_j \otimes \tau_{1,1} \otimes \id_{n-j-2} \right)
			 - q^{(n)}_{\lbrace\lbrace 1,\dots,i \rbrace, \lbrace i+1,\dots, n \rbrace\rbrace}\right. \\
\left|\ 0\leqslant i \leqslant n-1,\ 0\leqslant j \leqslant n-2,\ j\ne i-1
			   \right\rbrace.
	\end{split}\end{equation*}
	Moreover  $c_{m,n}(M) = 0$ for $m\ne n$, $c_{0,0}(M)=c_{1,1}(M)=1$ and for $n\geqslant 2$ we have
	 $$c_{n,n}(M) = \frac{1}{2}\binom{2n}{n} \sim \frac{ 4^n}{2\sqrt{\pi n}}
	\text{ as }n\to\infty.$$
\end{theorem}	
\begin{proof} By~\eqref{EqYDdim2VarnothingYDThetaRho}, the image of every pure braid
	may change only the sign of a tensor product of several $a$ and $b$. Therefore, such images commute
	and $\theta\left(\sigma_i^2 \rho - \rho \sigma_i^2 \right)\equiv 0$ holds in $M$.
	An explicit check 
	shows that the other elements of the set $V_3$ are polynomial identities in $M$ too.
	
	Denote the images of morphisms from $\mathcal P(\mathbbm k, \varnothing)(n,n)$ in $\mathcal P_{V_3}(\mathbbm k, \varnothing) (n,n)$ by the same symbols. Then $\tau_{1,1}^4=\id_2$ implies that both $q^{(n)}_{ij}$
	and $r^{(n)}_{ij}$ are idemponents and $$q^{(n)}_{ij}r^{(n)}_{ij}=r^{(n)}_{ij}q^{(n)}_{ij}=0.$$
	Now from $\theta\left(\sigma_i^2\right) \theta\left(\rho\right) =  \theta\left(\rho\right) \theta\left(\sigma_i^2\right)$
	for every $1\leqslant i \leqslant n-1$ and $\rho\in PB_n$ it follows that the images of every two pure braids commute and $q^{(n)}_{\lbrace I, \bar I\rbrace}$  are orthogonal idempotents for $\lbrace I, \bar I\rbrace \in \tilde \Xi_n$.
	
	The equality $q^{(n)}_{ij} \theta(\rho) = \theta(\rho) q^{(n)}_{ij}$ for $\rho\in B_n$,
	such that $\bar\rho(i)=i$ and $\bar\rho(j)=j$, implies
	 $r^{(n)}_{ij} \theta(\rho) = \theta(\rho) r^{(n)}_{ij}$ since $r^{(n)}_{ij} = \id_n - q^{(n)}_{ij}$.
	 In particular, $q^{(n)}_{ij}$ and $r^{(n)}_{ij}$ do not depend on the braid $\rho_{ij}$ used in their definition.
	Moreover, $q^{(n)}_{\lbrace I, \bar I \rbrace} \theta(\rho) = \theta(\rho) q^{(n)}_{\lbrace I, \bar I \rbrace}$
	for all $\lbrace I, \bar I \rbrace \in \tilde \Xi_n$ and $\rho\in B_n$ such that $\bar\rho(I)= I$.
	Finally, $\theta(\rho)q^{(n)}_{\lbrace J, \bar J\rbrace} \theta(\rho)^{-1} = q^{(n)}_{\lbrace I, \bar I \rbrace}$ for every $\lbrace I, \bar I\rbrace, \lbrace J, \bar J\rbrace \in \tilde \Xi_n$ and $\rho\in B_n$ such that $\bar\rho(J)= I$.
	Hence \begin{equation}\label{EqYDdim2VarnothingYDIJDifferent}q^{(n)}_{\lbrace I, \bar I\rbrace} \theta(\rho) q^{(n)}_{\lbrace J, \bar J\rbrace} =
	q^{(n)}_{\lbrace I, \bar I\rbrace} \theta(\rho)q^{(n)}_{\lbrace J, \bar J\rbrace} \theta(\rho)^{-1} \theta(\rho)  =
	q^{(n)}_{\lbrace I, \bar I\rbrace} q^{(n)}_{\lbrace J, \bar J\rbrace} \theta(\rho)^{-1}=0\end{equation}
	if $\bar \rho(J)\ne I$ and $\bar \rho(J)\ne \bar I$.

Let $\lbrace I, \bar I\rbrace, \lbrace J, \bar J\rbrace \in \tilde \Xi_n$.	
Consider the vector space $\left\langle\left. q^{(n)}_{\lbrace I, \bar I\rbrace} \theta(\rho) q^{(n)}_{\lbrace J, \bar J\rbrace}
\right| \rho\in B_n\right\rangle_{\mathbbm k}$. By~\eqref{EqYDdim2VarnothingYDIJDifferent},
it this space is zero if $|I|\ne |J|$ and $|I|\ne n-|J|$. Suppose $|I| = |J|=: i$. (If $|I|= n-|J|$, we just swap $J$ and $\bar J$.)
Choose $\rho_1, \rho_2\in B_n$ such that $\bar\rho_1(I)=\bar\rho_2(J)=\lbrace 1, \ldots, i\rbrace$.
Then \begin{equation*}\begin{split}\theta(\rho_1)\left\langle\left. q^{(n)}_{\lbrace I, \bar I\rbrace} \theta(\rho) q^{(n)}_{\lbrace J, \bar J\rbrace}
\right| \rho\in B_n\right\rangle_{\mathbbm k} \theta(\rho_2)^{-1} \\= 
\left\langle\left. \theta(\rho_1)q^{(n)}_{\lbrace I, \bar I\rbrace} \theta(\rho_1)^{-1} \theta\left(\rho_1\rho\rho_2^{-1}\right)\theta(\rho_2) q^{(n)}_{\lbrace J, \bar J\rbrace} \theta(\rho_2)^{-1}
\right| \rho\in B_n\right\rangle_{\mathbbm k} \\
=\left\langle\left. q^{(n)}_{\lbrace\lbrace 1,\ldots, i\rbrace, \lbrace i+1,\ldots, n\rbrace\rbrace} \theta(\rho) q^{(n)}_{\lbrace\lbrace 1,\ldots, i\rbrace, \lbrace i+1,\ldots, n\rbrace\rbrace}\right| \rho\in B_n\right\rangle_{\mathbbm k}\end{split}
\end{equation*}
If $\bar\rho(\lbrace 1,\ldots, i\rbrace) \ne \lbrace 1,\ldots, i\rbrace$ 
and $\bar\rho(\lbrace 1,\ldots, i\rbrace) \ne \lbrace i+1,\ldots, n\rbrace$, then, by~\eqref{EqYDdim2VarnothingYDIJDifferent}, $q^{(n)}_{\lbrace\lbrace 1,\ldots, i\rbrace, \lbrace i+1,\ldots, n\rbrace\rbrace} \theta(\rho) q^{(n)}_{\lbrace\lbrace 1,\ldots, i\rbrace, \lbrace i+1,\ldots, n\rbrace\rbrace} = 0$.
If $\bar\rho(\lbrace 1,\ldots, i\rbrace) = \lbrace 1,\ldots, i\rbrace$,
then $\theta(\rho)$ is a product of an image of a pure braid, which can be expressed via $q^{(n)}_{k\ell}$ and $r^{(n)}_{k\ell}$, and elements $\id_j \otimes \tau_{1,1} \otimes \id_{n-j-2} $
for $j\ne i-1$, which by $V_3$ do not alter $q^{(n)}_{\lbrace\lbrace 1,\ldots, i\rbrace, \lbrace i+1,\ldots, n\rbrace\rbrace}$.
If $\bar\rho(\lbrace 1,\ldots, i\rbrace) = \lbrace i+1,\ldots, n\rbrace$, then
$\rho=\rho_0 \rho_1$ where $\rho_0$ is a fixed braid such that $\bar\rho_0(\lbrace 1,\ldots, i\rbrace) = \lbrace i+1,\ldots, n\rbrace$ and $\rho_1$ is a braid such that $\bar\rho_1(\lbrace 1,\ldots, i\rbrace) = \lbrace 1,\ldots, i\rbrace$. Now the argument above can be applied to $\rho_1$.
Hence we get that
$$\dim\left\langle\left. q^{(n)}_{\lbrace I, \bar I\rbrace} \theta(\rho) q^{(n)}_{\lbrace J, \bar J\rbrace}
\right| \rho\in B_n\right\rangle_{\mathbbm k} $$ is zero if $|I|\ne |J|$ and $|I|\ne n-|J|$,
does not exceed $1$ if $|I| = |J| \ne \frac{n}{2}$ and does not exceed~$2$ if $|I| = |J| = \frac{n}{2}$.

For every $\lbrace I, \bar I\rbrace, \lbrace J, \bar J\rbrace \in \tilde \Xi_n$,
$|I| = |J|$,
fix some braid $\rho_{\lbrace I, \bar I\rbrace}\in B_n$ such that
$\bar\rho_{\lbrace I, \bar I\rbrace}(J)=I$.
In addition, for every $\lbrace I, \bar I\rbrace, \lbrace J, \bar J\rbrace \in \tilde \Xi_n$,
$|I| = |J| = \frac{n}{2}$,
fix some braid $\xi_{\lbrace I, \bar I\rbrace}\in B_n$ such that
$\bar\xi_{\lbrace I, \bar I\rbrace}(J)=\bar I$.

Inserting $\id_n = \sum\limits_{\lbrace I, \bar I\rbrace \in \tilde \Xi_n} q^{(n)}_{\lbrace I, \bar I\rbrace}$ and using the argument above, we obtain that
$\mathcal P_{V_3}(\mathbbm k, \varnothing) (n,n)$ is
the linear span of the $V_3$-relative $\varnothing$-polynomials
\begin{equation}\label{EqBasisYDkC2nothingYD} \begin{split}
q^{(n)}_{\lbrace I, \bar I\rbrace} \theta(\rho_{\lbrace I, \bar I\rbrace}) q^{(n)}_{\lbrace J, \bar J\rbrace}
\text{ for }\lbrace I, \bar I\rbrace, \lbrace J, \bar J\rbrace \in \tilde \Xi_n,\ |I|=|J|,\\
q^{(n)}_{\lbrace I, \bar I\rbrace} \theta(\xi_{\lbrace I, \bar I\rbrace}) q^{(n)}_{\lbrace J, \bar J\rbrace}
\text{ for }\lbrace I, \bar I\rbrace, \lbrace J, \bar J\rbrace \in \tilde \Xi_n,\ |I| = |J| = \frac{n}{2}. 
\end{split}\end{equation}

By~\eqref{EqYDdim2VarnothingYDThetaRho} and~\eqref{EqYDdim2nothingYDQIbarI}, the $V_3$-relative $\varnothing$-polynomials~\eqref{EqBasisYDkC2nothingYD} form a basis in $\mathcal P_{V_3}(\mathbbm k, \varnothing) (n,n)$, the restriction of $\mathcal E_M$ on $\mathcal P_{V_3}(\mathbbm k, \varnothing) (n,n)$ is an injection and
$V_3$
generates  $\varnothing$-polynomial identities in $M$
with coefficients in~$\mathbbm k$.

For $n=2k+1$ we have
$$c_{n,n}(M) = \sum\limits_{i=0}^k  \binom{2k+1}{i}^2= \frac{1}{2}
\sum\limits_{i=0}^n  \binom{n}{i}^2
= \frac{1}{2}\binom{2n}{n} \sim \frac{ 4^n}{2\sqrt{\pi n}}
\text{ as }n\to\infty,$$
for $n=2k$
we have
\begin{equation*}\begin{split}c_{n,n}(M) = 2\left(\frac{1}{2} \binom{2k}{k} \right)^2 + \sum\limits_{i=0}^{k-1}  \binom{2k}{i}^2=
 \frac{1}{2}
 \sum\limits_{i=0}^n  \binom{n}{i}^2 \\
 = \frac{1}{2}\binom{2n}{n} \sim \frac{ 4^n}{2\sqrt{\pi n}}
 \text{ as }n\to\infty.\end{split}\end{equation*}
\end{proof}

\section*{Acknowledgements}

The author is grateful to the referee for helpful remarks.


\begin{thebibliography}{99}

\bibitem{AGV2} Agore, A.\,L., Gordienko, A.\,S., Vercruysse, J.
$V$-universal Hopf algebras (co)acting on $\Omega$-algebras. \textit{Commun. Contemp. Math.},
\textbf{25}:1 (2023), 2150095-1 - 2150095-40.

\bibitem{AljaKassel} Aljadeff, E., Kassel, C. Polynomial identities and noncommutative versal torsors. \textit{Adv. Math.}, \textbf{218}:5 (2008), 1453--1495.

\bibitem{AnqCortMont} Anquella, J., Cort\'es, T., Montaner, F. Nonassociative coalgebras. \textit{Comm. Algebra}, \textbf{22}:12 (1994), 4693--4716.

\bibitem{Bahturin} Bakhturin, Yu.\,A.
Identical relations in Lie algebras.
VNU Science Press, Utrecht, 1987.

\bibitem{BahturinLinchenko} Bahturin, Yu.\,A., Linchenko, V.
Identities of algebras with actions of Hopf algebras. \textit{J. Algebra}, \textbf{202}:2 (1998), 634--654.

\bibitem{BereleHopf} Berele, A. Cocharacter sequences for algebras
with Hopf algebra actions. {\itshape J. Algebra}, \textbf{185}
(1996), 869--885.


\bibitem{DNR} D\u asc\u alescu, S., N\u ast\u asescu, C., Raianu, \c S.
Hopf algebras: an introduction. New York, Marcel Dekker, Inc., 2001.


\bibitem{DrenKurs} Drensky, V.\,S.
Free algebras and PI-algebras: graduate course in algebra.
Singapore, Springer-Verlag, 2000.

\bibitem{EGNObook} Etingof, P., Gelaki, S., Nikshych, D., Ostrik, V. Tensor categories.
\textit{AMS Mathematical Surveys and Monographs.} \textbf{205},
Providence, R.I., 2015, 362~pp.
%


\bibitem{GiaZai} Giambruno, A., Zaicev, M.\,V.
Polynomial identities and asymptotic methods.
\textit{AMS Mathematical Surveys and Monographs.} \textbf{122},
Providence, R.I., 2005, 352 pp.

\bibitem{ASGordienko9} Gordienko, A.\,S. Co-stability of radicals and its applications to PI-theory.
\textit{Algebra Colloquium}, \textbf{23}:3 (2016), 481--492.

\bibitem{ASGordienko15}
Gordienko, A.\,S. Actions of Ore extensions
and growth of polynomial $H$-identities. \textit{Comm. in Algebra},
\textbf{46}:7 (2018), 3014--3032.

\bibitem{KasselQuantumGroups} Kassel, C. Quantum groups. \textit{Graduate Texts in Mathematics.}
\textbf{155}, Springer-Verlag, New York, 1995, 533 pp.

\bibitem{KochetovCoalgId} Kochetov, M.\,V. On identities for coalgebras and Hopf algebras.
\textit{Comm. Algebra}, \textbf{28}:3 (2000), 1211--1221.

\bibitem{Lawvere} Lawvere, F.\,W. Functorial semantics of algebraic theories. \textit{Proc. Nat. Acad.
Sci. U.S.A.}, \textbf{50} (1963), 869--873.

\bibitem{Lin} Lin, I.-P. B. Semiperfect coalgebras. \textit{J. Algebra}, \textbf{30} (1974), 559--601.

\bibitem{MacLane65} Mac Lane, S. Categorical algebra. \textit{Bulletin of the AMS}, \textbf{71}:1 (1965), 40--106.

\bibitem{MarklPROPDef} Markl, M. Cotangent cohomology of a category and deformations. 
\textit{J. of Pure and Applied Algebra}, \textbf{113} (1996), 195--218.


\bibitem{Montgomery} Montgomery, S. Hopf algebras and their actions on rings, CBMS Lecture Notes \textbf{82}, Amer. Math. Soc., Providence, RI, 1993.


\bibitem{PirashviliPROPBialgebra} Pirashvili, T. On the PROP corresponding to bialgebras. \textit{Cahiers de Topologie et G\'eom\'etrie Diff\'erentielle Cat\'egoriques}, \textbf{43}:3 (2002), 221--239.

\bibitem{Porst1}
Porst, H.-E. The formal theory of Hopf algebras. Part I: Hopf monoids in a monoidal category.  \textit{Quaestiones Mathematicae}, \textbf{38}:5 (2015), 631--682.

\bibitem{RegevStrip} Regev, A. Asymptotic values for degrees associated with strips of Young diagrams. \textit{Adv. Math.}, 
\textbf{41} (1981), 115--136.

\bibitem{SweedlerBook} Sweedler M.E. Hopf Algebras, Benjamin, New York, 1969.

\end{thebibliography}
\end{document}